\documentclass[10pt,oneside]{amsart}
\usepackage{amsfonts,amsthm,amsmath,amssymb,mathtools}
\usepackage[all,cmtip]{xy}
\usepackage{soul} 
\usepackage[english]{babel}
\usepackage[utf8]{inputenc}
\usepackage[OT1]{fontenc}
\usepackage{fancyhdr}
\usepackage{mathrsfs}
\usepackage{cancel}
\usepackage{xcolor} 
\usepackage{enumitem} 
\usepackage[colorlinks=true, citecolor=blue, backref=none]{hyperref}  

\usepackage[layout=a4paper,right=2cm,left=2cm]{geometry}
\numberwithin{equation}{section}

\theoremstyle{plain}
\newtheorem{theorem}{Theorem}[section]
\newtheorem*{theorem*}{Theorem}
\newtheorem{proposition}[theorem]{Proposition}
\newtheorem*{proposition*}{Proposition}
\newtheorem{lemma}[theorem]{Lemma}
\newtheorem*{lemma*}{Lemma}
\newtheorem{corollary}[theorem]{Corollary}
\makeatletter
\newtheorem*{intr@thm}{\intr@thmname}
\newenvironment{introtheorem}[1]{%
	\def\intr@thmname{Theorem #1}
	\begin{intr@thm}}
	{\end{intr@thm}}
\newtheorem*{intr@crl}{\intr@crlname}
\newenvironment{introcorollary}[1]{%
	\def\intr@crlname{Corollary #1}
	\begin{intr@crl}}
	{\end{intr@crl}}
\makeatother
\newtheorem*{herbrand}{Herbrand quotient of units}

\theoremstyle{remark}
\newtheorem{remark}[theorem]{Remark}

\newtheorem*{acknowledgment}{Acknowledgment}
\newtheorem*{notation}{Notation}

\theoremstyle{definition}

\newtheorem*{definition*}{Definition}

\newcommand{\df}{:=}
\newcommand{\pd}[1]{#1\!\left[\tfrac{1}{2}\right]} 
\newcommand{\dsyl}[1]{#1_2} 
\newcommand{\Gal}{\mathrm{Gal}} 
\newcommand{\Z}{\mathbb{Z}} 
\newcommand{\N}{\mathbb{N}} 
\newcommand{\Q}{\mathbb{Q}} 
\newcommand{\rint}[1]{\mathcal{O}_{#1}} 
\newcommand{\unit}[1]{\rint{#1}^\times} 
\newcommand{\reg}[1]{R_{#1}} 
\newcommand{\clg}[1]{Cl_{#1}} 
\newcommand*{\arnm}[2]{\mathrm{N}_{#1/#2}}
\newcommand*{\arext}[2]{\iota_{#1/#2}}
\newcommand{\cln}[1]{h_{#1}} 
\newcommand{\clgp}[1]{\pd{\clg{#1}}} 
\newcommand{\princ}[1]{P_{#1}} 
\newcommand{\ideal}[1]{I_{#1}} 
\newcommand{\idele}[1]{\mathbb{A}^\times_{#1}} 
\newcommand{\idunit}[1]{\mathcal{U}_{#1}} 
\newcommand{\idclg}[1]{\mathcal{C}_{#1}} 
\newcommand{\q}[1]{Q_{#1}} 
\DeclarePairedDelimiterXPP{\hh}[3]{\widehat{h}^{#1}}{(}{)}{}{#2,#3} 
\DeclarePairedDelimiterXPP{\hhN}[3]{h^{#1}}{(}{)}{}{#2,#3} 
\DeclarePairedDelimiterXPP{\herbr}[2]{\vartheta}{(}{)}{}{#1,#2} 
\DeclarePairedDelimiterXPP{\HH}[3]{\widehat{H}^{#1}}{(}{)}{}{#2,#3} 
\DeclarePairedDelimiterXPP{\HHN}[3]{H^{#1}}{(}{)}{}{#2,#3} 
\DeclarePairedDelimiterXPP{\HhN}[3]{H_{#1}}{(}{)}{}{#2,#3} 
\DeclareMathOperator{\res}{res} 
\makeatletter
\def\r@otsl#1#2#3#4{\mu_{#4}(#2)}
\def\r@ots#1{\mu_{#1}}
\def\roots{\futurelet\radici\r@otschose}
\def\r@otschose{%
	\ifx\radici[
		\expandafter\r@otsl
	\else
		\expandafter\r@ots
\fi}
\makeatother 
\newcommand{\defect}[2]{\beta_{#1}(#2)} 
\DeclareMathOperator{\rank}{rk} 
\DeclarePairedDelimiterX\gindex[2]{(}{)}{#1:#2} 
\DeclarePairedDelimiter\gorder{\lvert}{\rvert} 
\newcommand{\kernel}[1]{\operatorname{Ker} #1} 
\newcommand{\cokernel}[1]{\operatorname{Coker} #1} 
\newcommand{\im}[1]{\operatorname{Im} #1}
\newcommand{\tors}[2][q]{{{#2}(#1)}}
\DeclareMathOperator{\tor}{tor} 
\newcommand{\ie}{i.~e.~} 
\newcommand{\rsp}{resp.~} 
\newcommand{\cfr}{cf.~} 

\addto\captionsenglish{}
\title{Class Number Formula for Dihedral Extensions}

\author[L.~Caputo]{Luca Caputo} 
\email{luca.caputo@gmx.com}
\author[F.~A.~E.~Nuccio]{Filippo A.~E.~Nuccio Mortarino Majno di Capriglio}
\email{filippo.nuccio@univ-st-etienne.fr}
\address{Univ Lyon, Université Jean Monnet Saint-Étienne, CNRS UMR 5208, Institut Camille Jordan, F-42023 Saint-Étienne, France}

\date{\today}


\begin{document}

\begin{abstract}
We give an algebraic proof of a class number formula for dihedral extensions of number fields of degree $2q$, where $q$ is any odd integer. Our formula expresses the ratio of class numbers as a ratio of orders of cohomology groups of units and allows one to recover similar formulas which have appeared in the literature. As a corollary of our main result we obtain explicit bounds on the (finitely many) possible values which can occur as ratio of class numbers in dihedral extensions. Such bounds are obtained by arithmetic means, without resorting to deep integral representation theory.
\end{abstract}
\maketitle
\section{Introduction}
Let $k$ be a number field and let $M/k$ be a finite Galois extension. For a subgroup $H\subseteq \Gal(M/k)$, let $\chi_{H}$ denote the rational character of the permutation representation defined by the $\Gal(M/k)$-set $\Gal(M/k)/H$. If $M/k$ is non-cyclic, then there exists a relation
\[
\sum_{H}n_H\cdot \chi_H=0
\]   
where $n_H\in \Z$ are not all $0$ and $H$ runs through the subgroups of $\Gal(M/k)$. Using the formalism of Artin $L$-functions and the formula for the residue at $1$ of the Dedekind zeta-function, Brauer showed in \cite{Bra} that the above relation translates into the following formula
\begin{equation}\label{eq:brauerformula}
\prod_{H } (\cln{M^H})^{n_H}=\prod_{H} \left(\frac{w_{M^H}}{\reg{M^H}}\right)^{n_H}
\end{equation}
where, for a number field $E$, $\cln{E}$, $\reg{E}$, $w_E$ denote the class number, the regulator and the order of the group of roots of unity of $E$, respectively.

Although the regulator of a number field is in general an irrational number, the right-hand side of the above equality is clearly rational, being equal to the left-hand side. As a matter of fact, Brauer himself proved that when $k=\Q$ the regulator can be translated into an algebraic invariant of the Galois structure of the unit group. It is therefore natural to look for an algebraic proof of Brauer's formula, independent of any analytical result. Brauer also showed that the left-hand side of \eqref{eq:brauerformula} can take only finitely many values as $M$ ranges over all Galois extension of $\Q$ with Galois group isomorphic to a fixed finite group. It thus becomes interesting to look for an explicit description of a finite set of integers (as small as possible) containing all the values which are attainable by the class number ratio, as well as for generalizations beyond the case $k=\Q$.

Historically, these questions were investigated for the first time in the case where $M=\Q(\sqrt{d},\sqrt{-d})$ (with $d>1$ a square-free integer). In this case Dirichlet showed, by analytical means, that
\[
\frac{\cln{M}}{\cln{\Q(\sqrt{d})}\cln{\Q(\sqrt{-d})}} = \frac{1}{2}Q
\]
where $Q$ is $1$ or $2$ depending on whether a certain unit of $M$ has norm $1$ or not in $\Q(\sqrt{-1})$ (see \cite{Dir}). Hilbert later gave an algebraic proof of the above equality (see \cite{Hil} and \cite{LemKuroda} for a more general result).

Another interesting particular case is that of dihedral extensions, which is the one we focus on in this paper. This setting has already been considered by other authors: notably, Walter computed explicitly in \cite{Wal77} the regulator term, building upon Brauer's previous work and obtained a refined class number formula in the dihedral case over an arbitrary base field $k$, still relying on analytical results about Artin $L$-functions (Walter's result actually holds for more general Frobenius groups). Bartel provided a different representation-theoretic interpretation of Walter's result in \cite{Ba}, a result which was later generalized to a much larger family of groups in \cite{BadeS}.

Our main result is the following theorem whose proof is purely algebraic and does not rely on any argument involving $L$-functions.
\begin{introtheorem}{\ref{yap}}
Let $L/k$ be a Galois extension of number fields whose Galois group \mbox{$D=\Gal(L/k)$} is dihedral of order $2q$ with $q$ odd. Let $\Sigma\subseteq D$ be a subgroup of order $2$ and $G\subseteq D$ be the subgroup of order $q$. Set $K=L^\Sigma$ and $F=L^G$. Then
\[
\frac{\cln{L}\cln{k}^2}{\cln{F}\cln{K}^2}=\frac{\hh*{0}{D}{\unit{L}\otimes\pd{\Z}}}{\hh*{-1}{D}{\unit{L}\otimes\pd{\Z}}} = \frac{\hh{0}{D}{\unit{L}}}{\hh{-1}{D}{\unit{L}}}\frac{\hh{-1}{\Sigma}{\unit{L}}}{\hh{0}{\Sigma}{\unit{L}}}.
\]
\end{introtheorem}
Here $\unit{L}$ is the group of units of $L$ and, for a subgroup $H\subseteq D$ and an integer $i$, $\hh{i}{H}{B}$ denotes the order of the $i$-th Tate cohomology group of $H$ with values in the $H$-module $B$.

Similar algebraic proofs have appeared in the literature although, as far as we know, only for special cases. Halter-Koch in \cite{HK} and Moser in \cite{Mos79} analysed the case of a dihedral extension of $\Q$ of degree $2p$, where $p$ is an odd prime. They express the ratio of class numbers in terms of a unit index, proving that in this setting one has
\[
\frac{\cln{L}}{\cln{F}\cln{K}^2}=\gindex*{\unit{L}}{\unit{K}\unit{K'}\unit{F}} p^{-r}
\]
where $K'$ is a Galois conjugate of $K$ in $L$, and $r=1$ if $F$ is imaginary and $2$ otherwise. Their result was generalized by Jaulent in \cite{Ja1}, who gave an algebraic proof of a similar formula for dihedral extensions of degree $2p^s$ when the ground field $k$ is $\Q$ or an imaginary quadratic field and $s\geq 1$ is any integer. Jaulent actually obtained his result as a particular case of a formula for Galois extensions over an arbitrary ground field with Galois group of the form $\Z/p^s \rtimes T$ where $T$ is an abelian group of exponent dividing $p-1$. In such general formula the ratio of class numbers is expressed in terms of orders of cohomology groups of units, as in our Theorem \ref{yap}. This approach was not developed further and in fact Jaulent's paper is not cited in subsequent works in the literature.

The study of class number formulas for dihedral extensions from an algebraic viewpoint was taken up more than 20 years later by Lemmermeyer. In \cite{Le}, he generalized Halter-Koch's formula to arbitrary ground fields, always under the assumption that the degree is $2p$ and under some ramification condition.

The quoted works also provide an explicit description of the finitely many possible values which may occur as class number ratio, clearly only under their working assumptions. These values are obtained using the explicit classification of the indecomposable integral representations of $D$, in order to find an exhaustive list of the possible indexes of units which can appear as regulator ratios. This classification is not available for dihedral groups of arbitrary order $2q$, as discussed in the introduction of \cite{HR2}, for instance. Combining our formula with the knowledge of the Herbrand quotient of the units for the cyclic subextension $L/F$, we find a finite set of integers containing all the attainable values of the ratio of class numbers, without resorting to any integral representation theoretic results. This indicates that the traditional way of expressing the class number ratio as a unit index, although suggestive, is less effective than our cohomological approach for the purpose of determining the possible values of the ratio of class numbers.

Our main results in this direction are the following two corollaries. For a number field $M$, denote by $\roots[{q^\infty}]{M}$ the roots of unity in $M$ of $q$-power order, and by $\roots[{q}]{M}$ the subgroup of $\roots[{q^\infty}]{M}$ of elements killed by $q$. Set
\[
\defect{M}{q}=
\begin{cases}
0&\text{if $\roots[q]{M}$ is trivial,}\\
1&\text{otherwise}
\end{cases}
\]   	
Give a a prime number $\ell$,  $v_\ell$ denotes the $\ell$-adic valuation.
\begin{introcorollary}{\ref{cor:bounds}}
Let $a=\rank_\Z(\unit{F})+ \defect{F}{q}+1$ and $b=\rank_\Z(\unit{k})+\defect{k}{q}$. Then, for every prime~$\ell$, the following bounds hold
\[
-av_\ell(q)\leq v_\ell\left(\frac{\cln{L}\cln{k}^2}{\cln{F}\cln{K}^2}\right)\leq bv_\ell(q).
\]
In particular, if $F$ is totally real of degree $[F\colon\Q]=2d$, then
\[
-2dv_\ell(q)\leq v_\ell\left(\frac{\cln{L}\cln{k}^2}{\cln{F}\cln{K}^2}\right)\leq (d-1)v_\ell(q). 
\]
\end{introcorollary}

In case $F$ is assumed to be a CM field and $k=F^+$ is its totally real subfield, we can strengthen our bounds as follows.

\begin{introcorollary}{\ref{cor:boundsCM}}
Assume that $F$ is a CM field of degree $[F:\Q]=2d$ and that $k$ is its totally real subfield. Let $s$ be the order of the quotient $\bigl((\unit{K})^q\cap\unit{k}\bigr)/(\unit{k})^q$ and let $t$ be the order of $\roots[{q^\infty}]{F}/\roots[q]{F}$. Then, for every prime~$\ell$,
\[
-v_\ell(q)-v_\ell(s)-v_\ell(t)\leq v_\ell\left(\frac{\cln{L}\cln{k}^2}{\cln{F}\cln{K}^2}\right)\leq (d-1)v_\ell(q).
\]
\end{introcorollary}
When $q=p$ is prime and $k=\Q$, we can further improve Corollary \ref{cor:bounds} and prove in Corollary \ref{cor:boundsQ} that the ratio of class numbers satisfies
\[
\frac{\cln{L}}{\cln{F}\cln{K}^2}\in\left\{1,\frac{1}{p},\frac{1}{p^2}\right\}.
\]
This answers the MathOverflow question \cite{BartelMO} of Bartel, who asks for a proof of this fact that does not rely on the analytical class number formula, nor on any integral representation theory.  Our Corollary \ref{cor:bounds} actually shows that a similar result holds without the assumption that $k=\Q$ or that $q=p$ be prime.
  
We now briefly describe the structure of the paper. Section \ref{sec:cohomprelim} contains all the technical algebraic lemmas used in the rest of the paper. The proof of Theorem \ref{yap}, which is divided in two parts, is contained in Section \ref{sec:classnumberformula}. The proof of the $2$-part follows a strategy of Walter (see \cite{Wal77}), combining an elementary result about $D$-modules and a simple arithmetic piece of information. On the other hand, the proof for the odd components heavily relies on cohomological class field theory and is partly inspired by the proof of Chevalley's ambiguous class number formula. Since the extension $L/k$ is non-abelian we cannot apply class field theory directly: however, since we are restricted to odd parts, the action of $\Delta=\Gal(F/k)$ on the cohomology of $G=\Gal(L/F)$ can be described very explicitly as discussed in Section \ref{sec:cohomprelim}. We conclude the article with Section \ref{sec:equivalence_unit_index} where we translate our cohomological formula in an explicit unit index (see Corollary \ref{cor:formulawithunitindex}).

\begin{acknowledgment} 
This paper has a long history. We started working on it in 2007, during a summer school in Iwasawa theory held at McMaster University, following a suggestion of Ralph Greenberg. At the time our main focus was on what we called ``fake $\Z_p$-extensions''  and the formula for dihedral extensions was simply a tool which we obtained by translating the right-hand side of \eqref{eq:brauerformula} into a unit index essentially using linear algebra. We posted a version of our work (quite different from the present one) on arXiv in 2009 and it became part of the second author's PhD thesis \cite{Nuc09}. At that time we were confident that our techniques would easily generalize to cover other classes of extensions, so we decided to wait. Then basically nothing happened until the autumn of 2016 when the second author came across a question \cite{BartelMO} posted by Alex Bartel on MathOverflow. This motivated us to revise and improve our work, and to split the original version in two articles. This is the first of the two papers, concentrating on the class number formula for finite dihedral extensions, while in the second, which is still in progress, we focus on fake $\Z_p$-extensions. We are grateful to Henri Johnston for encouraging us to take up this revision task and to Alex Bartel for some comments on a preliminary version. We also thank the anonymous referee for useful comments and suggestions which improved the presentation of our paper.
\end{acknowledgment}
	
\section{Cohomological preliminaries}\label{sec:cohomprelim}
This section contains all the technical algebraic tools we need in the rest of the paper. We fix an odd positive integer $q\geq 3$ and we let $D$ denote the dihedral group of order $2q$. We choose generators $\sigma,\rho\in D$ such that 
	\[
	\sigma^2=\rho^{q}=1, \quad \sigma\rho = \rho^{-1} \sigma.
	\]
We set $\Sigma \df \langle\sigma\rangle$, $\Sigma'\df \langle\sigma\rho\rangle$ and $G\df\langle\rho\rangle,\Delta\df D/G$. We say that an abelian group~$B$ is uniquely~$2$-divisible when multiplication by $2$ is an automorphism of~$B$. As we shall see below, the~$D$-Tate cohomology group $\HH{i}{D}{B}$ of a uniquely $2$-divisible~$D$-module~$B$ is canonically isomorphic to the subgroup of $\HH{i}{G}{B}$ on which $\Delta$ acts trivially. We find a similar cohomological description for the subgroup $\HH{i}{G}{B}$ on which $\Delta$ acts as $-1$: the precise statement is given in Proposition~\ref{fixcohom}. 
 
	\begin{notation} 
Given a group $\Gamma$ and a $\Gamma$-module $B$, we denote by $B^\Gamma$ (\rsp $B_\Gamma$) the maximal submodule (\rsp the maximal quotient) of $B$ on which $\Gamma$ acts trivially. We will sometimes refer to $B^\Gamma$ (\rsp $B_\Gamma$) as the invariants (\rsp coinvariants) of $B$. Moreover, let $N_\Gamma=\sum_{\gamma\in\Gamma}\gamma\in\Z[\Gamma]$ be the norm, $B[N_\Gamma]$ the kernel of multiplication by $N_\Gamma$ and $I_{\Gamma}$ the augmentation ideal of $\Z[\Gamma]$ defined as
\[
I_\Gamma\df\kernel(\Z[\Gamma]\longrightarrow \Z)\qquad\text{ where the morphism is induced by }\gamma\mapsto 1\text{ for all }\gamma\in\Gamma.
\]
For every $i\in \Z$, let $\HH{i}{\Gamma}{B}$ denote the $i$th Tate cohomology group of $\Gamma$ with values in $B$. Similarly, for $i\geq 0$, $\HHN{i}{\Gamma}{B}$ (\rsp $\HhN{i}{\Gamma}{B}$) is the $i$th cohomology (\rsp homology) group of $\Gamma$ with values in $B$. For standard properties of these groups (which will be implicitly used without specific mention) the reader is referred to \cite[Chapter~XII]{CE}. If $B'$ is another $\Gamma$-module and $f\colon B\to B'$ is a homomorphism of abelian groups, we say that $f$ is $\Gamma$-equivariant (\rsp $\Gamma$-antiequivariant) if $f(\gamma b)=\gamma f(b)$ (\rsp $f(\gamma b)=-\gamma f(b)$) for every $\gamma\in\Gamma$ and $b\in B$. Suppose now that $\Gamma=\{1,\gamma\}$ has order $2$. An object $B$ of the category of uniquely $2$-divisible $\Gamma$-modules admits a functorial decomposition $B=B^+\oplus B^-$ where we denote by $B^+$ (\rsp $B^-$) the maximal submodule on which $\Gamma$ acts trivially (\rsp as $-1$). Such decomposition is obtained by writing $b=\frac{1+\gamma}{2}b+\frac{1-\gamma}{2}b$ for $b\in B$. Then 
\begin{equation}\label{2divid}
B^\Gamma=B^+\cong B/B^-=B_\Gamma.
\end{equation}
\end{notation}

We now go back to our setting and concentrate on the case where $\Gamma$ is equal to a subquotient of $D$. Observe that, if $B$ is a uniquely $2$-divisible $D$-module, then $B^H$ is again uniquely $2$-divisible for every subgroup $H\subseteq D$. In particular, for every $i\in \Z$, the map
\[
\HH{i}{\Delta}{B^G} \overset{\cdot 2}{\longrightarrow} \HH{i}{\Delta}{B^G}
\]
is an automorphism and therefore 
\begin{equation}\label{eq:2divct}
\HH{i}{\Delta}{B^{G}}=0
\end{equation}
because these groups are killed by $2=\gorder{\Delta}$. Using this we prove the next proposition, which is the key of all our cohomological computations.

	\begin{proposition}\label{fixcohom} 
		Let $B$ be a uniquely $2$-divisible $D$-module. Then, for every $i\in \Z$, restriction induces isomorphisms 
		\[
		\HH{i}{D}{B}\cong \HH{i}{G}{B}^+.
		\]
		Moreover, the Tate isomorphism $\HH{i}{G}{B}\cong \HH{i+2}{G}{B}$ is $\Delta$-antiequivariant. In particular, for every $i\in\Z$, there are isomorphisms
		\begin{align*}
		\HH{i}{G}{B}^-&\cong\HH{i+2}{G}{B}^+\cong \HH{i+2}{D}{B}\\ 
		\HH{i}{G}{B}^+&\cong\HH{i+2}{G}{B}^-\cong \HH{i}{D}{B}\end{align*}
		of abelian groups. 
	\end{proposition}
	\begin{proof} 
		We start by proving the first isomorphisms for $i\in\Z$ with $i\ne 0,-1$, \ie for standard cohomology and homology. Since $\HHN{i}{\Delta}{B^G}=\HhN{i}{\Delta}{B^G}=0$ by \eqref{eq:2divct}, using the Hochschild--Serre spectral sequence (see, for instance, \cite[Chapter~XVI, \S~6]{CE}), we deduce that restriction induces isomorphisms $\HHN{i}{D}{B}\cong \HHN{i}{G}{B}^\Delta = \HHN{i}{G}{B}^+$ and $\HhN{i}{D}{B}\cong \HhN{i}{G}{B}_{\Delta}= \HhN{i}{G}{B}^{+}$ for $i\geq 0$. 

We now consider the case $i=0$. Restriction induces a commutative diagram with exact rows
		\begin{equation}\begin{aligned}\label{eq:ND=NG+}
	\xymatrix{
	0\ar@{->}[0,1]&N_DB\ar@{->}[0,1]\ar@{->}[1,0]&\HHN{0}{D}{B}\ar@{->}[0,1]\ar@{->}[1,0]&\HH{0}{D}{B}\ar@{->}[0,1]\ar@{->}[1,0]&0\\
	0\ar@{->}[0,1]&(N_GB)^+\ar@{->}[0,1]&\HHN{0}{G}{B}^+\ar@{->}[0,1]&\HH{0}{G}{B}^+\ar@{->}[0,1]&0
	}
	\end{aligned}\end{equation}
The bottom row is exact since $N_GB$ is uniquely $2$-divisible, hence $\Delta$-cohomologically trivial as shown in \eqref{eq:2divct}. We showed above that the central vertical arrow is an isomorphism. In particular the right vertical arrow is surjective. It is also injective since restriction followed by corestriction is multiplication by $2$ which is an automorphism of $\HH{0}{D}{B}$. Therefore
\[
\HH{0}{G}{B}^+\cong \HH{0}{D}{B}
\]
as claimed. The proof is similar in degree $-1$, by writing $\HH{-1}{D}{B}=\kernel\bigl(N_D:\HhN{0}{D}{B}\to N_DB\bigr)$ and similarly for $G$.
		
Finally, we discuss the $\Delta$-antiequivariance of Tate isomorphisms. Recall that, for any $i\in\Z$, the Tate isomorphism is given by the cup product with a fixed generator $\chi$ of $\HHN{2}{G}{\Z}$:
\begin{equation*}\begin{array}{ccc}
\HH{i}{G}{B}&\longrightarrow&\HH{i+2}{G}{B}\\
x&\longmapsto&x\cup \chi.
\end{array}\end{equation*}
As in \cite[Chap.~X,~\S~7]{CE}, we let $\delta\in\Delta$ be the non-trivial element and we denote by $c_\delta$ the automorphism induced by $\delta$ on $G$-Tate cohomology of $B$. Then $c_\delta$ is multiplication by $-1$ on $\HHN{2}{G}{\Z}$: indeed, since $\delta$ acts as~$-1$ on $G$, it also acts as $-1$ on $\operatorname{Hom}(G,\Q/\Z)$. Our claim follows by considering the isomorphism of $\Delta$-modules $\operatorname{Hom}(G,\Q/\Z)=\HHN{1}{G}{\Q/\Z}\cong\HHN{2}{G}{\Z}$ given by the connecting homomorphism of the exact sequence
\[
0\longrightarrow \Z \longrightarrow \Q \longrightarrow\Q/\Z\to 0 
\]
of $G$-modules with trivial action. Then, by \cite[Chap.~XII, \S 8, (14)]{CE}, 
\[
c_\delta(x\cup\chi)=-(c_\delta x)\cup \chi
\] 
which establishes the result.
\end{proof}

\begin{remark}
The above proposition implies in particular that the Tate cohomology of a uniquely $2$-divisible $D$-module is periodic of period $4$. Indeed, if $B$ is such a module and $i\in\Z$, we have isomorphisms
\[
\HH{i}{D}{B} \cong \HH{i}{G}{B}^+ \cong \HH{i+2}{G}{B}^- \cong \HH{i+4}{G}{B}^+\cong\HH{i+4}{D}{B}.
\]
In fact, one can prove that the cohomology of \emph{every} $D$-module is periodic of period $4$ (\cfr \cite[Theorem XII.11.6]{CE}). 
\end{remark}

\begin{remark}
If $B$ is a uniquely $2$-divisible $D$-module such that $\HH{j}{G}{B}=0$ for some $j\in \Z$, then
	\[
	\HH{j+2i}{D}{B}=0\qquad\text{ for all }i\in\Z.
	\]
Indeed, we have seen in the above proposition that $\HH{j+2i}{D}{B}\subseteq \HH{j+2i}{G}{B}$. On the other hand, the periodicity of Tate cohomology for the cyclic group $G$ implies that $\HH{j+2i}{G}{B}\cong \HH{j}{G}{B}=0$.
\end{remark}

We now give a description of some cohomology groups of $D$ with values in a $D$-module $B$ in terms of explicit subquotients of $B$. We start with the following lemma.
	
\begin{lemma}\label{lemma:2incl}
Let $B$ be a $D$-module. Then 
\[
I_{G}B\subseteq B^{\Sigma}+B^{\Sigma'}.
\]
\end{lemma}
\begin{proof}
This is shown in \cite[Lemma 3.4]{Le} and we recall here the details of the proof. For every $b\in B$, we have
\[
(1-\rho)b = - \bigl((1+\sigma)\rho\bigr)b+ (1+\sigma\rho)b.
\] 
It follows that
\[
I_{G}B = (1-\rho)B \subseteq \bigl((1+\sigma)\rho\bigr)B+ (1+\sigma\rho)B=\bigl(1+\sigma\bigr)B+ (1+\sigma\rho)B=N_\Sigma B+N_{\Sigma'}B\subseteq B^{\Sigma}+B^{\Sigma'}.
\] 
\end{proof}

Let $B$ be a $D$-module. Observe that $B^{\Sigma'}=\rho^{t}(B^{\Sigma})$ for any $t\in\Z$ such that $2t \equiv -1 \pmod{q}$. Indeed, one easily checks that $\rho^t(B^\Sigma) = B^{\rho^t \Sigma \rho^{-t}}$. Now, observing that $\rho^t \sigma \rho^{-t} = \sigma\rho^{-2t}=\sigma\rho$, we deduce that $\rho^t\Sigma\rho^{-t}=\langle\sigma\rho\rangle=\Sigma'$ and therefore $B^{\Sigma'}=\rho^{t}(B^{\Sigma})$.
In particular, a generic element $b_1 + b_2\in B^{\Sigma}+B^{\Sigma'}$ can be written as
\[
b_1+b_2=b_1+\rho^t(b') = b_1 + b' - (1 -\rho^t)b' \in B^{\Sigma} + I_{G}B
\]
for some $b'\in B^\Sigma$. Using Lemma \ref{lemma:2incl}, we deduce that
\begin{equation}\label{eq:BSigma'}
B^{\Sigma} + B^{\Sigma'} = B^{\Sigma} + I_{G}B = B^{\Sigma'} + I_{G}B
\end{equation}
and, in particular, if $B$ is uniquely $2$-divisible,
\begin{equation}\label{eq:BSigma-}
\bigl(B^{\Sigma} + B^{\Sigma'} \bigr)^-= (I_{G}B)^-
\end{equation}
since $\Sigma$ acts trivially on $\bigl(B^{\Sigma} + I_GB\bigr)/I_GB$.
	
	\begin{lemma}\label{lemma:quozcoom}
		Let $B$ be a uniquely $2$-divisible $D$-module. There are isomorphisms of abelian groups
		\begin{align*}
		\Bigl(B[N_G]\cap \bigl(B^{\Sigma}+B^{\Sigma'}\bigr)\Bigr)\Big/I_{G}B&\cong \HH{-1}{D}{B}\\
		\intertext{and}
		B[N_{G}]\Big/\Bigl(\bigl(B^{\Sigma}+B^{\Sigma'}\bigr)\cap B[N_G]\Bigr)&\cong \HHN{1}{D}{B}.
		\end{align*}
	\end{lemma}
	\begin{proof}
		Consider the short exact sequence defining $\HH{-1}{G}{B}$:
		\[
		0\longrightarrow  I_{G}B\longrightarrow  B[N_{G}] \longrightarrow  \HH{-1}{G}{B}\longrightarrow  0.
		\]
		Since $B$ is uniquely $2$-divisible, taking $\Sigma$-invariants and using Proposition \ref{fixcohom}, we get an exact sequence
		\begin{equation}\label{eq:sesH-1DB}
		0\longrightarrow  I_{G}B\cap B^{\Sigma}\longrightarrow  B[N_{G}]\cap B^{\Sigma} \longrightarrow \HH{-1}{D}{B}\longrightarrow  0.
		\end{equation}
		Since $I_{G}B\subseteq B^{\Sigma}+B^{\Sigma'}$ by Lemma \ref{lemma:2incl}, we can consider the quotient $\bigl(B[N_{G}]\cap(B^{\Sigma}+B^{\Sigma'})\bigr)\big/I_{G}B$, as well as the map
		\[
		\bigl(B[N_{G}]\cap B^{\Sigma}\bigr)\big/\bigl(I_{G}B\cap B^{\Sigma}\bigr) \longrightarrow\Bigl(B[N_{G}]\cap \bigl(B^{\Sigma}+B^{\Sigma'}\bigr)\Bigr)\big/I_{G}B
		\]
		induced by the inclusion $B[N_{G}]\cap B^{\Sigma} \subseteq B[N_{G}]\cap \bigl(B^{\Sigma}+B^{\Sigma'}\bigr)$. The above map is clearly injective and is also surjective since
		\begin{equation*}\label{eq:bsigmaigb}
		B^{\Sigma} + I_{G}B = B^{\Sigma} + B^{\Sigma'}
		\end{equation*}
		by \eqref{eq:BSigma'}. This, together with the exact sequence in \eqref{eq:sesH-1DB}, proves the first isomorphism of the lemma.
		
		For the second isomorphism, note that, by Proposition \ref{fixcohom}, we have an isomorphism of abelian groups
		\[
		\HHN{1}{D}{B} \cong \HH{-1}{G}{B}/ \HH{-1}{D}{B}.
		\qedhere
		\]
\end{proof}

Observe that the proof of Lemma \ref{lemma:quozcoom} actually shows that
\begin{equation}\label{eq:quozcoom}
\HH{-1}{G}{B}^+ = \Bigl(B[N_G]\cap \bigl(B^{\Sigma}+B^{\Sigma'}\bigr)\Bigr)\Big/I_{G}B \quad \textrm{and}\quad \HH{-1}{G}{B}^- = B[N_{G}]\Big/\Bigl(\bigl(B^{\Sigma}+B^{\Sigma'}\bigr)\cap B[N_G]\Bigr).
\end{equation}

Another consequence of Lemma \ref{lemma:2incl} is the following.
\begin{lemma}\label{lemma:split_Walter}
Let $B$ be a uniquely $q$-divisible finite $D$-module. There is a direct-sum decomposition
\[
B/B^D\cong B^G/B^D\oplus \bigl(B^{\Sigma}+B^{\Sigma'}\bigr)/B^D
\]
which implies, upon taking the quotient by $B^G/B^D$,
\[
B/B^G\cong B^\Sigma/B^D\oplus B^{\Sigma'}/B^D.
\]
\end{lemma}
\begin{proof}
We are thankful to the referee for suggesting us the following proof, which simplifies the one in our original version.

One can easily check that there is a relation
\[
N_D - N_G - \sum_{\begin{subarray}{l}\widetilde{\Sigma}< D\\ \vert\widetilde{\Sigma}\rvert=2\end{subarray}}N_{\widetilde{\Sigma}} = -q
\]
in $\Z[D]$. In particular, since multiplication by $q$ is an automorphism of $B$, we have
\[
B = \biggl(N_D - N_G - \sum_{\begin{subarray}{l}\widetilde{\Sigma}< D\\ \vert\widetilde{\Sigma}\rvert=2\end{subarray}}N_{\widetilde{\Sigma}}\biggr)B \subseteq B^G+\sum_{\begin{subarray}{l}\widetilde{\Sigma}< D\\ \vert\widetilde{\Sigma}\rvert=2\end{subarray}} B^{\widetilde{\Sigma}} \subseteq B,
\]
so that in fact
\begin{equation}\label{eq:Bdecomp}
B = B^G+\sum_{\begin{subarray}{l}\widetilde{\Sigma}< D\\ \vert\widetilde{\Sigma}\rvert=2\end{subarray}}B^{\widetilde{\Sigma}}.
\end{equation}
Now we claim that 
\begin{equation}\label{eq:allintwo}
\sum_{\begin{subarray}{l}\widetilde{\Sigma}< D\\ \vert\widetilde{\Sigma}\rvert=2\end{subarray}}B^{\widetilde{\Sigma}}= B^\Sigma + B^{\Sigma'}.
\end{equation}
To see this, it is enough to show that for every subgroup $\widetilde{\Sigma}\subset D$ of order $2$, we have $B^{\widetilde{\Sigma}}\subseteq B^{\Sigma} + B^{\Sigma'}$. Let $\widetilde{\Sigma}=\langle\sigma\rho^{i}\rangle$ where either $i=2j$ or $i=2j+1$ for some $j\in\Z$ and let $b\in B^{\widetilde{\Sigma}}$. Then $\rho^jb\in B^{\rho^j \widetilde{\Sigma}\rho^{-j}}$ and $\rho^j \widetilde{\Sigma}\rho^{-j} = \langle\sigma \rho^{i-2j}\rangle$ is either $\Sigma$ or $\Sigma'$, according as $i$ is even or odd. Thus
\[
b=(1-\rho^j)b + \rho^jb \in I_GB + (B^{\Sigma} + B^{\Sigma'}) =B^{\Sigma} + B^{\Sigma'}
\]
by Lemma \ref{lemma:2incl}. Therefore \eqref{eq:allintwo} holds and \eqref{eq:Bdecomp} becomes
\[
B = B^G+ (B^{\Sigma} + B^{\Sigma'}).
\] 
To show that this sum is direct modulo $B^D$, let $b\in B^G\cap  (B^{\Sigma} + B^{\Sigma'})$, so that $b=b_1+b_2$ with $b_1\in B^\Sigma$ and $b_2\in B^{\Sigma'}$. Then
\[
b = \frac{1}{q}N_Gb = \frac{1}{q}N_Gb_1 + \frac{1}{q}N_Gb_2,
\]
which implies $b\in B^D$ since $N_Gb_i \in B^D$ for $i=1,2$.
\end{proof}

To use the above results in the next sections, we systematically tensor the module of interest with $\pd{\Z}$. When no ambiguity is possible, if $f\colon B\to B'$ is a homomorphism of abelian groups we will still denote by $f$ the map $\pd{f}\colon\pd{B}\to \pd{B'}$.
	
We constantly use  that, for every $D$-module $B$ and every subgroup $H\subseteq D$, there are isomorphisms
	\[
	\HH*{i}{H}{\pd{B}}\cong \pd{\HH{i}{H}{B}}
	\]
for all $i\in\Z$. This is quite straightforward, but we sketch an argument, writing
\[
\Lambda=\Z[H]\qquad\text{and}\qquad\Lambda'=\pd{\Z[H]}.
\]
The fact that $\pd{\Z}$ is $\Z$-flat implies that $\pd{\HHN{0}{H}{B}}=\pd{\operatorname{Hom}_{\Lambda}(\Z,B)}$ coincides with $\operatorname{Hom}_{\Lambda'}(\pd{\Z},\pd{B})$. On the other hand, a $\Lambda'$-linear homomorphism from $\pd{\Z}$ to $\pd{B}$ is uniquely determined by its value on $1$, as is every $\Lambda$-homomorphism from $\Z$ to $\pd{B}$, showing that $\operatorname{Hom}_{\Lambda'}(\pd{\Z},\pd{B})=\operatorname{Hom}_{\Lambda}(\Z,\pd{B})=\HHN{0}{H}{\pd{B}}$. Therefore the right-derived functors of the functor $\mathcal{F}_1\colon B\mapsto\HHN{0}{H}{\pd{B}}$ and of $\mathcal{F}_2\colon B\mapsto \pd{\HHN{0}{H}{B}}$ coincide. Those of the first are $R^i\mathcal{F}_1(B)= \HHN{i}{H}{\pd{B}}$, and those of the second are $R^i\mathcal{F}_2(B)=\pd{\HHN{i}{H}{B}}$, for $i\geq 0$: this follows from a computation with the Grothendieck spectral sequence of a composition of functors, using that $-\otimes \pd{\Z}$ is acyclic by assumption. It follows that $\HHN{i}{H}{\pd{B}}$ and $\pd{\HHN{i}{H}{B}}$ are isomorphic. A similar argument is valid for $H$-homology. The remaining cases of Tate cohomology in degrees $-1,0$ can be treated directly by writing the Tate cohomology groups as a kernel and a cokernel, respectively (\cfr \cite[Chap.~XII, \S~2]{CE}).

For an abelian group $B$ and a natural number $n$, we set $\tors[n]{B}=\{b\in B : nb=0\}$. 

\begin{lemma}\label{lemma:schifo=nucleo} 
Let $B$ be $D$-module such that $\tors{B}=\tors{B}^G$. Then the inclusion $\tors{B}\subseteq B$ induces equalities
\[
\kernel\bigl(\HH{-1}{G}{\tors{B}}\longrightarrow\HH{-1}{G}{B}\bigr)=I_GB\cap \tors{B}^G=I_GB\cap B^G.
\]
If, moreover, $B$ is uniquely $2$-divisible then
\[
\kernel\bigl(\HH{-1}{D}{\tors{B}}\longrightarrow\HH{-1}{D}{B}\bigr)=I_GB\cap \tors{B}^D=I_GB\cap B^D.
\]
\end{lemma}
\begin{proof} 
Since $\tors{B}$ has trivial action of $G$ by assumption, $I_G(\tors{B})=0$ and $N_G$ is multiplication by $q$ on $\tors{B}$ and so is the trivial map. Hence $\HH{-1}{G}{\tors{B}}=\tors{B}$ and we obtain the first equality of the statement since
\[
\kernel\bigl(\HH{-1}{G}{\tors{B}}\longrightarrow\HH{-1}{G}{B}\bigr)= \kernel\bigl(\tors{B}\longrightarrow B[N_G]/I_GB\bigr)=I_GB \cap \tors{B}=I_GB \cap \tors{B}^G.
\]
As for the second, there is an obvious inclusion $I_GB\cap \tors{B}^G\subseteq I_GB\cap B^G$. To prove the reverse inclusion it is enough to observe that every $x\in I_GB\cap B^G$ satisfies $qx=N_Gx$ (because $x\in B^G$) but also $N_Gx=0$ (since $x\in I_GB\subseteq B[N_G]$). Thus we have $x\in I_GB\cap B^G\cap \tors{B}=I_GB\cap \tors{B}^G$.

When $B$ is uniquely $2$-divisible, by taking plus parts in the above equalities and applying Proposition \ref{fixcohom}, we obtain the second claim.
\end{proof}

We conclude this section with two lemmas which give some relations between the cohomology of $D$ and that of its subgroups, without the assumption that the underlying module is uniquely $2$-divisible.

\begin{lemma}\label{lemma:Sylowtrivialcohom}
Let $H\subseteq D$ be a subgroup, let $B$ be an $H$-module and let $j\in\Z$ be an integer. Suppose that, for every prime $p$, $\HH{j}{H_p}{B}=0$ where $H_p$ is a $p$-Sylow subgroup of $H$. Then $\HH{i}{H}{B}=0$ for every $i\equiv j \pmod{2}$.
\end{lemma}
\begin{proof}
Under our assumptions, for every prime $p$, $\HH{i}{H_p}{B}$ is trivial for every $i\equiv j \pmod{2}$ since $p$-Sylow subgroups of $H$ are cyclic and the Tate cohomology of cyclic groups is periodic of period $2$. Since the restriction map
\[
\HH{i}{H}{B} \longrightarrow \HH{i}{H_p}{B}
\]
is injective on the $p$-primary component of $\HH{i}{H}{B}$ we deduce that $\HH{i}{H}{B}=0$ for every $i\equiv j \pmod{2}$.
\end{proof}	

\begin{lemma}\label{lemma:restr2Sylow}
Let $B$ a $D$-module. Then, for every $i\in\Z$, the restriction 
\[
\res_\Sigma^D\colon\HH{i}{D}{B} \longrightarrow \HH{i}{\Sigma}{B}
\]
maps the $2$-primary component of $\HH{i}{D}{B}$ isomorphically onto $\HH{i}{\Sigma}{B}$.
\end{lemma}
\begin{proof}
		As observed in the proof of Lemma \ref{lemma:Sylowtrivialcohom}, we only need to show surjectivity. The image of the $2$-component is the subset of stable elements of $\HH{i}{\Sigma}{B}$ (see \cite[Chapter XII,~Theorem~10.1]{CE}). Recall that an element $a\in\HH{i}{\Sigma}{B}$ is called stable if, for every $\tau \in D$,
		\[
		\res_{\tau\Sigma\tau^{-1}\cap \Sigma}^\Sigma(a)=\res_{\tau\Sigma\tau^{-1}\cap \Sigma}^{\tau\Sigma\tau^{-1}}c_\tau(a)
		\]
		where $c_\tau$ denotes the map $\HH{i}{\Sigma}{B}\to\HH{i}{\tau\Sigma\tau^{-1}}{B}$ induced by conjugation by $\tau$. When $\tau\notin\Sigma$, then $\tau\Sigma\tau^{-1}\cap \Sigma$ is trivial and the above equality holds for every $a\in \HH{i}{\Sigma}{B}$. When $\tau\in\Sigma$, then $\tau\Sigma\tau^{-1}=\Sigma$ and the two restriction maps appearing in the above equality are in fact the identity map. The same holds for $c_\tau$ (see \cite[Chapter~XII, \S~8, (5)]{CE}), proving that every element in $\HH{i}{\Sigma}{B}$ is stable. 
	\end{proof}

\section{Class number formula}\label{sec:classnumberformula}
We now come to the proof of the odd part of the class number formula. Still under the notation introduced in the previous section, we consider the following arithmetic setting. Let $k$ be a number field and let $L/k$ be a Galois extension whose Galois group is isomorphic to $D$. We fix such an isomorphism and we simply write $D=\Gal(L/k)$. Let $F/k$ be the subextension of $L/k$ fixed by $G$ and let $K/k$ (\rsp $K'/k$) be the subextension of $L/k$ fixed by $\Sigma$ (\rsp $\Sigma'$). In particular $F/k$ is quadratic with Galois group $\Delta\df\Gal(F/k)$.
	
We begin by introducing a four-term exact sequence involving the class groups of $L$ and $K$. By looking at the orders of groups in this sequence, we already get a formula involving the odd components of class numbers of $L$ and $K$. The rest of the section is devoted to expressing the remaining factors in terms of the odd components of the class numbers of $F$ and $k$ and the orders of the cohomology group of units of $L$. As a corollary of the formula, using the value of the Herbrand quotient of units for the extension $L/F$, we easily obtain bounds on the ratio of class numbers of subfields of $L$.     

\begin{notation}	
For a number field $M$, we write $\clg{M}$ for the class group of $M$. For an extension of number fields $M_2/M_1$, $\arext{M_2}{M_1}\colon\clg{M_1}\to \clg{M_2}$ and $\arnm{M_2}{M_1}\colon\clg{M_2}\to \clg{M_1}$ denote the map induced by extending ideals from $M_1$ to $M_2$ and by taking the norm of ideals from $M_2$ to $M_1$, respectively.
\end{notation}

Consider the exact sequence
	\begin{equation}\label{nuccioskernelgen}
	1\to \kernel{\eta}\longrightarrow \clg{K} \oplus \clg{K'} \stackrel{\eta}{\longrightarrow} \clg{L} \longrightarrow \clg{L}/\im{\eta}\to 1
	\end{equation}
	where the map $\eta$ is defined via the formula
	\begin{equation}\label{eq:nucciomap}
	\eta\left(c,c'\right)=\arext{L}{K}(c)\arext{L}{K'}(c') \qquad \text{ for }c\in \clg{K}, c'\in \clg{K'}.
	\end{equation}
 Tensoring with $\pd{\Z}$, we immediately obtain the formula
\begin{equation}\label{eq:formula_cl_fixed}
\gorder*{\clgp{L}} = \gorder*{\clgp{K}}^2 \frac{\gindex*{\clgp{L}}{\im{\pd{\eta}}}}{\gorder*{ \kernel{\pd{\eta}} }}.
\end{equation}
Observe also that the map $\pd{\arext{L}{K}}\colon\clgp{K}\to\clgp{L}$ is injective, since $\arnm{L}{K}\circ\arext{L}{K}$ is multiplication by $2=[L:K]$ on $\clg{K}$.
	
\begin{remark}
Let $B$ be an abelian group and $B'\subseteq B$ a subgroup of $B$. Let $f\colon B\to C$ be a homomorphism of abelian groups. Then there is an exact sequence
\[
0\longrightarrow \kernel{f}/\bigl(\kernel{f}\cap B'\bigr)\longrightarrow B/B'\longrightarrow f(B)/f(B')\longrightarrow 0.
\]
It follows in particular that
\begin{equation}\label{eq:snake}
\gindex{B}{B'}=\gindex{f(B)}{f(B')}\cdot\gindex{\kernel{f}}{\kernel{f}\cap B'},
\end{equation}
meaning that if two of the above indices are finite, so is the third and the above equality holds.
\end{remark}
	For a group $H$ and a $H$-module $B$, we denote by $\hh{i}{H}{B}$ (\rsp $\hhN{i}{H}{B}$) the order of $\HH{i}{H}{B}$ (\rsp the order of $\HHN{i}{H}{B}$), whenever this is defined.
	
	\begin{proposition}\label{loddpart}
		We have an equality
		\[
		\gindex*{\clgp{L}}{\pd{\im{\eta}}}  = \frac{\gorder*{\clgp{L}^G}}{\hh*{-1}{D}{\clgp{L}}\cdot \gorder*{N_D \clgp{L} }}
		\]
		and an isomorphism
		\[
		\pd{\kernel{\eta}}\cong \HHN*{0}{D}{\clgp{L}}.
		\] 
		In particular
		\[
		\gorder*{\clgp{L}} = \gorder*{\clgp{K}}^2 \frac{\gorder*{\clgp{L}^G}}{\hhN*{0}{D}{\clgp{L}}\hh*{-1}{D}{\clgp{L}}\cdot \gorder*{N_D \clgp{L} }}.
		\]
	\end{proposition}
	\begin{proof}
Observe that, using equation \eqref{eq:snake} applied to $f=N_G$, $B=\clg{L}$, $B'=\im{\eta}$, we get
	\[
	\gindex{\clg{L}}{\im{\eta}} =\gindex{N_G \clg{L}}{N_G(\im{\eta})}\cdot \gindex{\clg{L}[N_G]}{\im{\eta}\cap \clg{L}[N_G]} .
	\]
Now observe that 
	\[
	\im{\eta} = \arext{L}{K}(\clg{K})\arext{L}{K'}(\clg{K'})
	\]
and
\begin{equation}\label{eq:fixedequalincl}
\left(\clgp{L}\right)^{\Sigma}=\arext{L}{K}\bigl(\clgp{K}\bigr),\qquad\left(\clgp{L}\right)^{\Sigma'}=\arext{L}{K'}\bigl(\clgp{K'}\bigr)
\end{equation}
so that
\begin{equation*}
\pd{\im{\eta}}= \clgp{L}^{\Sigma}\clgp{L}^{\Sigma'}.
\end{equation*}
Therefore, using Lemma \ref{lemma:quozcoom}, we deduce that 
\begin{equation*}
\gindex*{\clgp{L}}{\pd{\im{\eta}}} =\gindex*{N_G \clgp{L} }{N_D \clgp{L} }\cdot \hhN*{1}{D}{\clgp{L}},
\end{equation*}
because
\[
N_D \clgp{L} =N_G\bigl(N_\Sigma \clgp{L} \bigr)=N_G\bigl(\clgp{L}^{\Sigma}\bigr)=N_G\bigl(\arext{L}{K}\bigl(\clgp{K}\bigr)\bigr)=N_G\bigl(\arext{L}{K'}\bigl(\clgp{K'}\bigr)\bigr),
\] 
since $\HH{0}{\Sigma}{\clgp{L}}=0$. Now observe that 
\[
\gorder*{N_G\clgp{L}} =\frac{\gorder*{\clgp{L}^G}}{\hh*{0}{G}{\clgp{L}}}.
\] 
The first equality of the proposition follows, since
\[
\frac{\hhN*{1}{D}{\clgp{L}}}{\hhN*{1}{G}{\clgp{L}}}=\frac{1}{\hh*{-1}{D}{\clgp{L}}}
\]
by Proposition \ref{fixcohom} and because the finiteness of $\pd{\clg{L}}$ implies $\hh{0}{G}{\pd{\clg{L}}} = \hhN{1}{G}{\pd{\clg{L}}}$.
	
As for the second statement of the proposition, let $\theta$ be the map
\[
\theta\colon\pd{\kernel{\eta}}\longrightarrow \HHN*{0}{D}{\pd{\clg{L}}}
\]
defined by $\theta(c,c')=\arext{L}{K}(c)$. Since $\arext{L}{K}(c)=\arext{L}{K'}(c')^{-1}$ in $\clg{L}$ by definition of $\eta$, both $\Sigma$ and $\Sigma'$ fix $\arext{L}{K}(c)$, so the map $\theta$ takes values in $\HHN{0}{D}{\clg{L}}$. We claim that it is an isomorphism. First, $\theta$ is injective because $\arext{L}{K}\colon\clgp{K}\to \clgp{L}$ is injective, as observed after equation \eqref{eq:formula_cl_fixed}. To check surjectivity, note that, since every $b\in \HHN{0}{D}{\pd{\clg{L}}}$ is fixed by both $\Sigma$ and $\Sigma'$, we can write such $b$ as $b=\arext{L}{K}(c)=\arext{L}{K'}(c')$ for some $c\in \pd{\clg{K}}$, $c'\in \pd{\clg{K'}}$, using \eqref{eq:fixedequalincl}. In particular $(c,(c')^{-1})\in\pd{\kernel{\eta}}$ and $b=\theta(c,(c')^{-1})$, which shows the claimed isomorphism and hence the second statement of the proposition.
		
		The last formula of the proposition comes from \eqref{eq:formula_cl_fixed} using the first two assertions.
	\end{proof}
	
We now want to translate cohomology groups of $\clgp{L}$ into cohomology groups of $\pd{\unit{L}}$. To achieve this we make intensive use of the following commutative diagram of Galois modules with exact rows and columns:
	\begin{equation}\begin{aligned}\label{diag:magic}
	\xymatrix{
	&1\ar@{->}[1,0]&1\ar@{->}[1,0]&1\ar@{->}[1,0]&\\
	1\ar@{->}[0,1]&\unit{M}\ar@{->}[0,1]\ar@{->}[1,0]&M^\times\ar@{->}[0,1]\ar@{->}[1,0]&\princ{M}\ar@{->}[0,1]\ar@{->}[1,0]&1\\
	1\ar@{->}[0,1]&\idunit{M}\ar@{->}[0,1]\ar@{->}[1,0]&\idele{M}\ar@{->}[0,1]\ar@{->}[1,0]&\ideal{M}\ar@{->}[0,1]\ar@{->}[1,0]&1\\
	1\ar@{->}[0,1]&\q{M}\ar@{->}[0,1]\ar@{->}[1,0]&\idclg{M}\ar@{->}[0,1]\ar@{->}[1,0]&\clg{M}\ar@{->}[0,1]\ar@{->}[1,0]&1\\
	&1&1&1
	}\end{aligned}\end{equation}
Here we use the following notation for a number field $M$: $\rint{M}$ is the ring of integers, $\ideal{M}$ is the group of fractional ideals, $\princ{M}$ is the group of principal ideals, $\idele{M}$ is the id\`eles group, $\idclg{M}$ is the id\`eles class group, $\idunit{M}$ is the group of id\`eles of valuation $0$ at every finite place and $\q{M}$ is defined by the exactness of the diagram. We refer the reader to \cite{Ta} for the class field theory results that we use. 

\begin{proposition}\label{prop:generale} 
Let $H\subseteq D$ be any subgroup and $j$ be an odd integer. Then
\[
\HH{j}{H}{\ideal{L}}=\HH{j}{H}{\idclg{L}}=\HH{j}{H}{L^\times}=\HH{j}{H}{\idele{L}}=0.
\]
\end{proposition}
\begin{proof}
By Lemma \ref{lemma:Sylowtrivialcohom}, we only need to prove that $\HHN{1}{H}{B}=0$ where $B$ is any of the modules $\ideal{L}, \idclg{L}, L^\times, \idele{L}$ and $H$ is cyclic. This is Hilbert's Theorem 90 when $B=L^\times, \idele{L}$ and follows from \cite[Theorem~9.1]{Ta} when $B=\idclg{L}$. For  $B=\ideal{L}$, let $M\subseteq L$ be the subfield of $L$ fixed by $H$. There is an $H$-equivariant isomorphism
	\begin{equation}\label{eq:ideali_galois}
	\ideal{L} \cong \bigoplus_{\mathfrak{l}\subseteq \rint{M}}\Z[H/H(\mathfrak{L})]
	\end{equation}
where $\mathfrak{l}$ runs over the primes of $M$ and $H(\mathfrak{L})$ denotes the decomposition group in $L/M$ of a fixed prime $\mathfrak{L}$ of $L$ above $\mathfrak{l}$. Hence, by Shapiro's lemma,
	\[
	\HHN{1}{H}{\ideal{L}} \cong \bigoplus_{\mathfrak{l}\subseteq \rint{M}} \HHN{1}{H(\mathfrak{L})}{\Z}. 
	\]
The statement follows, since $\HHN{1}{H(\mathfrak{L})}{\Z}=\mathrm{Hom}(H(\mathfrak{L}),\Z)=0$.
\end{proof}

We can now move towards the proof of our main result, contained in Theorem \ref{yap}. 

	\begin{proposition}\label{nucciocohom} 
We have isomorphisms
	\[
	\HH*{0}{D}{\pd{\clg{L}}}\cong \HHN*{1}{D}{\pd{\q{L}}}\qquad \text{and}
	\qquad\HH*{-1}{D}{\pd{\clg{L}}}\cong \HH*{0}{D}{\pd{\q{L}}}.
	\]
	\end{proposition}
	\begin{proof} 
From the long exact $D$-Tate cohomology sequence of the bottom row of Diagram \ref{diag:magic}, we get an exact sequence
\begin{equation}\label{number}
\HH{-1}{D}{\clg{L}}\longrightarrow\HH{0}{D}{\q{L}}\longrightarrow\HH{0}{D}{\idclg{L}}\longrightarrow\HH{0}{D}{\clg{L}}\longrightarrow\HHN{1}{D}{\q{L}}.
\end{equation}
The rightmost map is surjective and the leftmost map is injective thanks to Proposition \ref{prop:generale}. Tensoring the sequence with $\pd{\Z}$, the middle term of \eqref{number} becomes trivial, since class field theory gives an isomorphism
\begin{equation}\label{eq:H0TDidclg}
\HH{0}{D}{\idclg{L}}\stackrel{\cong}{\longrightarrow} D^{ab}\cong \Z/2\Z.\qedhere
\end{equation}
\end{proof}

The following proposition is the dihedral analogue of Chevalley's Ambiguous Class Number Formula (which holds for cyclic groups). Our proof is inspired by the cohomological proof of Chevalley's formula (see \cite[Lemma 4.1]{LaCF}). 
	
\begin{proposition}\label{servesotto} 
Let $H\subseteq D$ be a subgroup and let $M$ be the subfield of $L$ fixed by $H$. There is a long exact sequence
	\[
	0\longrightarrow \kernel{\arext{L}{M}}\longrightarrow \HHN{1}{H}{\unit{L}} \longrightarrow \HHN{1}{H}{\idunit{L}} \longrightarrow \HHN{0}{H}{\clg{L}}/\arext{L}{M}(\clg{M})\longrightarrow \HHN{1}{H}{\princ{L}} \longrightarrow 0.
	\]
In particular
\[
\frac{\hhN*{1}{H}{\pd{\idunit{L}}} \hhN*{1}{H}{\pd{\princ{L}}}}{\hhN*{0}{H}{\clgp{L}}}=\frac{\hhN*{1}{H}{\pd{\unit{L}}}}{\gorder*{\clgp{M}}}.
\]
\end{proposition}
\begin{proof}
From the rightmost column of Diagram \ref{diag:magic} we get a commutative diagram with exact rows
\[
\xymatrix{
0\ar@{->}[0,1]& \princ{M}\ar@{->}[0,1]\ar@{}[1,0]^/-1.1pc/{}="a"\ar@{}[1,0]^/1.25pc/{}="b"& \ideal{M}\ar@{->}[0,1]\ar@{}[1,0]^/-1.1pc/{}="c"\ar@{}[1,0]^/1.25pc/{}="d"& \clg{M} \ar@{->}[1,0]^{\arext{L}{M}} \ar@{->}[0,1]& 0\\
0\ar@{->}[0,1]& \HHN{0}{H}{\princ{L}}\ar@{->}[0,1]& \HHN{0}{H}{\ideal{L}}\ar@{->}[0,1]& \HHN{0}{H}{\clg{L}}\ar@{->}[0,1]& \HHN{1}{H}{\princ{L}}\ar[r]& 0.
\ar@{^(->}"a";"b"\ar@{^(->}"c";"d"}
\]
The exactness of the bottom row follows from Proposition \ref{prop:generale}. Applying the snake lemma we get the exact sequence
\begin{align}\label{illuminazione}
0\longrightarrow \kernel{\arext{L}{M}}\longrightarrow \HHN{0}{H}{\princ{L}}/\princ{M}&\longrightarrow \HHN{0}{H}{\ideal{L}}/\ideal{M}\longrightarrow\nonumber\\
&\longrightarrow \HHN{0}{H}{\clg{L}}/\arext{L}{M}(\clg{M})\longrightarrow \HHN{1}{H}{\princ{L}}\longrightarrow 0.
\end{align}
Now, from the upper row of Diagram \ref{diag:magic}, we get the exact sequence
\[
0\longrightarrow M^\times/\unit{M}= \princ{M} \longrightarrow \HHN{0}{H}{\princ{L}} \longrightarrow \HHN{1}{H}{\unit{L}} \longrightarrow 0.
\] 
We thus obtain an isomorphism
\begin{equation}\label{kernel}
\HHN{0}{H}{\princ{L}}/\princ{M} \cong \HHN{1}{H}{\unit{L}}.
\end{equation}
As for the middle term of \eqref{illuminazione}, consider the central row of Diagram \ref{diag:magic}. Using that $\HHN{1}{H}{\idele{L}}=0$ by Proposition \ref{prop:generale}, we get an exact sequence
\[
0\longrightarrow M^\times/\idunit{M}=\ideal{M}\longrightarrow \HHN{0}{H}{\ideal{L}}\longrightarrow \HHN{1}{H}{\idunit{L}}\longrightarrow 0.
\]
This gives an isomorphism
\begin{equation}\label{cokernel}
\HHN{0}{H}{\ideal{L}}/\ideal{M}\cong \HHN{1}{H}{\idunit{L}}.
\end{equation}
Plugging \eqref{kernel} and \eqref{cokernel} in the exact sequence \eqref{illuminazione}, we obtain the exact sequence of the statement. The formula for the orders follows since
\[
\gorder{\kernel{\arext{L}{M}}}\cdot\gorder{\arext{L}{M}(\clg{M})} =\gorder{\clg{M}}.\qedhere
\]
\end{proof}
\begin{remark}
Observe that the norm $\arnm{L}{k}\colon\pd{\idclg{L}}\to \pd{\idclg{k}}$ is surjective (for this remark only, we extend the notation $\arnm{L}{k}$ also to denote the corresponding morphism on id\`ele class groups and similarly for $\arext{L}{k}$). Indeed, Hilbert 90 implies that $\idclg{L}^{D}=\arext{L}{k}(\idclg{k})$ and in particular $\arext{L}{k}\colon\idclg{k}\to\idclg{L}$ is injective. Since $N_D \idclg{L}=\arext{L}{k}(\arnm{L}{k}(\idclg{L}))$, we deduce that
\[
\HH{0}{D}{\idclg{L}}= \idclg{L}^D/ N_D \idclg{L}= \arext{L}{k}(\idclg{k})/ \arext{L}{k}(\arnm{L}{k}(\idclg{L})) \cong \idclg{k}/\arnm{L}{k}(\idclg{L}).
\]
In particular, since $\HH{0}{D}{\pd{\idclg{L}}}=0$ by class field theory (see \eqref{eq:H0TDidclg}), we find that $\pd{\idclg{k}}=\arnm{L}{k}(\pd{\idclg{L}})$. Looking at the commutative diagram with exact rows
\[\xymatrix{
\idclg{L}\ar@{->}[0,1]\ar@{->}[1,0]^{\arnm{L}{k}}&\clg{L}\ar@{->}[0,1]\ar@{->}[1,0]^{\arnm{L}{k}}&0\\
\idclg{k}\ar@{->}[0,1]&\clg{k}\ar@{->}[0,1]&0
}\] 
we deduce that $\arnm{L}{k}\colon\pd{\clg{L}}\to \pd{\clg{k}}$ must be surjective too. Moreover, since $N_D \clg{L} = \arext{L}{k}(\arnm{L}{k}(\clg{L}))$, we find that 
\begin{equation}\label{eq:ND=ik}
N_D\pd{\clg{L}} = \arext{L}{k}\bigl(\pd{\clg{k}}\bigr)
\end{equation}
and in particular
\[
\HH*{0}{D}{\pd{\clg{L}}}= \HHN*{0}{D}{\pd{\clg{L}}}/\arext{L}{k}\bigl(\pd{\clg{k}}\bigr).
\]
\end{remark}

\begin{proposition}\label{prop:gessisullalavagna}
Let $H\subseteq D$ be any subgroup and let $M$ be the subfield of $L$ fixed by $H$. There is an exact sequence
\[
0\longrightarrow \HH{-1}{H}{\princ{L}}\longrightarrow \HH{0}{H}{\unit{L}}\longrightarrow \HH{0}{H}{\idunit{L}}\longrightarrow\HH{0}{H}{\q{L}}\longrightarrow \kernel{\arext{L}{M}}\longrightarrow 0.
\] 
In particular
\[
\frac{\hh*{0}{D}{\pd{\idunit{L}}}\hh*{-1}{D}{\pd{\princ{L}}}}{\hh*{-1}{D}{\clgp{L}}\cdot \gorder*{N_D \clgp{L}}}=\frac{\hh*{0}{D}{\pd{\unit{L}}}}{\gorder{\clgp{k}}}.
\]
\end{proposition}
\begin{proof}
From the long exact sequence of Tate cohomology of the left-hand column of Diagram \ref{diag:magic}, we get an exact sequence
\[
\HH{0}{H}{\unit{L}}\longrightarrow\HH{0}{H}{\idunit{L}}\longrightarrow\HH{0}{H}{\q{L}}\longrightarrow \HHN{1}{H}{\unit{L}}\longrightarrow \HHN{1}{H}{\idunit{L}}.
\]
By Proposition \ref{servesotto} we know that 
\[
\kernel\bigl(\HHN{1}{H}{\unit{L}}\to \HHN{1}{H}{\idunit{L}}\bigr) = \kernel\bigl(\arext{L}{M}\bigr).
\]
Now observe that
\begin{equation}\label{eq:ker=ker}
\kernel\bigl(\HH{0}{H}{\unit{L}}\to \HH{0}{H}{\idunit{L}}\bigr)=\kernel\bigl(\HH{0}{H}{\unit{L}}\to \HH{0}{H}{L^\times}\bigr).
\end{equation}
This follows from the commutative diagram
\[
\xymatrix{
\HH{0}{H}{\unit{L}}\ar[r]\ar[1,0] &\HH{0}{H}{L^\times}\ar@{}[1,0]^/-1.25pc/{}="a"\ar@{}[1,0]^/1.15pc/{}="b"\\
\HH{0}{H}{\idunit{L}} \ar@{^{(}->}[0,1] &\HH{0}{H}{\idele{L}}
\ar@{^(->}"a";"b"
}
\]
where the injectivity of bottom horizontal and right vertical arrows follow from Proposition \ref{prop:generale}. The exact sequence of the statement follows since
\[
\kernel\bigl(\HH{0}{H}{\unit{L}}\to \HH{0}{H}{L^\times}\bigr) = \HH{-1}{H}{\princ{L}}
\]
by Proposition \ref{prop:generale}. The last assertion follows by Proposition \ref{nucciocohom} and \eqref{eq:ND=ik}.	\end{proof}

\begin{remark}
Concerning the cokernel of $\arext{L}{M}$, similar techniques as in the above proof allow to give a cohomological interpretation. Taking $H$-cohomology of the bottom row of diagram \eqref{diag:magic} and recalling that $\HHN{1}{H}{\idclg{L}}=0$, we get a commutative diagram
\begin{equation*}
\xymatrix{
0\ar@{->}[0,1]&\q{M}\ar@{->}[1,0]\ar@{->}[0,1]&\idclg{M}\ar@{=}[1,0]\ar@{->}[0,1]&\clg{M}\ar@{->}[1,0]^{\arext{L}{M}}\ar@{->}[0,1]&0\\
0\ar@{->}[0,1]&\q{L}^H\ar@{->}[0,1]&\idclg{L}^H\ar@{->}[0,1]&\clg{L}^H\ar@{->}[0,1]&\HHN{1}{H}{\q{L}}\ar@{->}[0,1]&0
}\end{equation*}
Applying the snake lemma to the above diagram, we obtain an isomorphism
\begin{equation}\label{eq:coker_qcohom}
\cokernel{\arext{L}{M}}\cong \HHN{1}{H}{\q{L}}
\end{equation}
which will be useful when considering the $2$-part of our main result.
\end{remark}

We now give an explicit description of the $D$-cohomology of $\idunit{L}$, using local class field theory. We need the following notation. 
\begin{notation}
If $M$ is a subfield of $L$, we denote by $T(L/M)$ the set of primes in $M$ which ramify in $L/M$. We also write $T(L/k)=T_s(L/k)\coprod T_r(L/k)\coprod T_i(L/k)$ where the sets $T_s(L/k), T_r(L/k)$ and $T_i(L/k)$ contain the primes in $T(L/k)$ which, respectively, split, ramify and remain inert in the subextension $F/k$. For each prime $\mathfrak{l}$ of $k$ we denote by $e_\mathfrak{l}$ and $f_\mathfrak{l}$ its ramification index and inertia degree, respectively, in $L/k$. For a prime $\mathfrak{L}$ of $L$, we denote by $D(\mathfrak{L})$ the decomposition group of $\mathfrak{L}$ in $L/k$. Given a prime $\mathfrak{r}$ of $M\subseteq L$, we use the notation $\unit{M,\mathfrak{r}}$ instead of $\unit{M_{\mathfrak{r}}}$ to denote the units of the local field $M_\mathfrak{r}$, in order to lighten typography.
\end{notation}
	
We start with a lemma describing the cohomology of local units. 
\begin{lemma}\label{lemma:f_in_F}
Let $\mathfrak{l}\in T(L/k)$ be a prime and let $\mathfrak{L}\subseteq\rint{L}$ be a prime in $L$ above $\mathfrak{l}$. 
\begin{enumerate}[label=\alph*)]
\item If $\mathfrak{l}\in T_i(L/k)$, then $f_\mathfrak{l}=2$ and $\HH{0}{D(\mathfrak{L})}{\unit{L,\mathfrak{L}}}=0$;
\item If $\mathfrak{l}\in T_r(L/k)$, then $f_\mathfrak{l}=1$ and $\HH{0}{D(\mathfrak{L})}{\unit{L,\mathfrak{L}}}\cong \Z/2$;
\item If $\mathfrak{l}\in T_s(L/k)$, then $\HH{0}{D(\mathfrak{L})}{\unit{L,\mathfrak{L}}}\cong \Z/e_\mathfrak{l}$.
\end{enumerate}
Moreover, for any $\mathfrak{l}\in T(L/k)$, we have $\HHN{1}{D(\mathfrak{L})}{\unit{L,\mathfrak{L}}}\cong \Z/e_\mathfrak{l}$.
\end{lemma}
\begin{proof}
Consider the following diagram with exact rows
\[\xymatrix{
0\ar@{->}[0,1]&\unit{L,\mathfrak{L}}\ar@{->}[0,1]\ar@{->}[1,0]&L_\mathfrak{L}^\times\ar@{->}[0,1]^{v_\mathfrak{L}}\ar@{->}[1,0]&\Z\ar@{->}[0,1]\ar@{}[1,0]^/-1.25pc/{}="a"\ar@{}[1,0]^/1.25pc/{}="b"&0\\
0\ar@{->}[0,1]&\unit{k,\mathfrak{l}}\ar@{->}[0,1]&k_\mathfrak{l}^\times\ar@{->}[0,1]^{v_\mathfrak{l}}&\Z\ar@{->}[0,1]&0
\ar@{^(->}"a";"b"_{\cdot f_\mathfrak{l}}
}\]
where the left and central vertical maps are the norms and $v_\mathfrak{l}$ and $v_\mathfrak{L}$ are the $\mathfrak{l}$-adic and the $\mathfrak{L}$-adic valuation, respectively. The snake lemma, together with local class field theory, gives an exact sequence
\[
0\longrightarrow \HH{0}{D(\mathfrak{L})}{\unit{L,\mathfrak{L}}}\longrightarrow D(\mathfrak{L})^{\mathrm{ab}}\longrightarrow \Z/f_\mathfrak{l}\longrightarrow 0.
\]
If $\mathfrak{l}\in T_i(L/k)\cup T_r(L/k)$, then $D(\mathfrak{L})$ has order divisible by $2$ so it must be dihedral or cyclic of order $2$. In both cases $D(\mathfrak{L})^{\mathrm{ab}}\cong \Z/2$ and therefore $f_\mathfrak{l}$ divides $2$, by the above exact sequence. When $\mathfrak{l}\in T_i(L/k)$, we must have $f_\mathfrak{l}=2$ and $\HH{0}{D(\mathfrak{L})}{\unit{L,\mathfrak{L}}}=0$. When $\mathfrak{l}\in T_r(L/k)$, we deduce $f_\mathfrak{l}=1$, because otherwise $f_\mathfrak{l}e_\mathfrak{l}=\gorder{D(\mathfrak{L})}$ would be divisible by $4$ and this is absurd: so $\HH{0}{D(\mathfrak{L})}{\unit{L,\mathfrak{L}}}\cong \Z/2$. Finally, if $\mathfrak{l}\in T_s(L/k)$, then $D(\mathfrak{L})^{\mathrm{ab}}=D(\mathfrak{L})$ is cyclic and, since $f_\mathfrak{l}e_\mathfrak{l}=\gorder{D(\mathfrak{L})}$, we deduce that $\HH{0}{D(\mathfrak{L})}{\unit{L,\mathfrak{L}}}\cong \Z/e_\mathfrak{l}$.
		
To prove the assertion on $H^1$, apply the snake lemma to the commutative diagram with exact rows 
\[\xymatrix{
0\ar@{->}[0,1]&\unit{k,\mathfrak{l}}\ar@{->}[0,1]\ar@{=}[1,0]&k_\mathfrak{l}^\times\ar@{->}[0,1]^{v_\mathfrak{l}}\ar@{=}[1,0]&\Z\ar@{}[1,0]^/-1.25pc/{}="a"\ar@{}[1,0]^/1.25pc/{}="b"\ar@{->}[0,1] &0\\
0\ar@{->}[0,1]&\bigl(\unit{L,\mathfrak{L}}\bigr)^{D(\mathfrak{L})}\ar@{->}[0,1]&\bigl(L_\mathfrak{L}^\times\bigr)^{D(\mathfrak{L})}\ar@{->}[0,1]^/+1em/{v_\mathfrak{L}}&\Z\ar@{->}[0,1]&\HHN{1}{D(\mathfrak{L})}{\unit{L,\mathfrak{L}}}\ar@{->}[0,1]&\HHN{1}{D(\mathfrak{L})}{L_\mathfrak{L}^\times}=0
\ar@{^(->}"a";"b"_{\cdot e_\mathfrak{l}}}
\] 
to deduce an isomorphism
\[
\HHN{1}{D(\mathfrak{L})}{\unit{L,\mathfrak{L}}} \cong \Z/e_\mathfrak{l}.\qedhere
\]
\end{proof}
	
\begin{proposition}\label{prop:coomuni} 
With the notation introduced before Lemma \ref{lemma:f_in_F}, there are isomorphisms
\begin{align*}
\HH{0}{D}{\idunit{L}}\;\cong&\bigoplus_{\mathfrak{l}\in T_s(L/k)}
(\Z/e_\mathfrak{l})\;\oplus
\bigoplus_{\mathfrak{l}\in T_r(L/k)} (\Z/2)
\intertext{ and }
\HHN{1}{D}{\idunit{L}}\;\cong&
\bigoplus_{\mathfrak{l}\in T_s(L/k)}(\Z/e_\mathfrak{l})\;\oplus\bigoplus_{\mathfrak{l}\in T_i(L/k)}
(\Z/e_\mathfrak{l})\;\oplus\bigoplus_{\mathfrak{l}\in T_r(L/k)}
(\Z/e_\mathfrak{l}).
\end{align*}
\end{proposition}
\begin{proof}
For every prime $\mathfrak{l}$ of $k$ we fix a prime $\mathfrak{L}\mid\mathfrak{l}$ in $L$. By combining \cite[Proposition 7.2]{Ta} with the fact that local units are cohomologically trivial in unramified extension (see \cite[Chapter 1, Proposition 1]{SerCF}) we find 
\begin{equation}\label{eq:shapiro}
\HH{i}{D}{\idunit{L}}\cong\bigoplus_{\mathfrak{l}\subseteq\rint{k}}\HH{i}{D(\mathfrak{L})}{\unit{L,\mathfrak{L}}}\cong\bigoplus_{\mathfrak{l}\in T(L/k)}\HH{i}{D(\mathfrak{L})}{\unit{L,\mathfrak{L}}}\qquad \text{for all } i\in\Z.
\end{equation}
The proposition follows from Lemma \ref{lemma:f_in_F}.
\end{proof}

Recall that, by Proposition \ref{fixcohom}, we have an isomorphism of $\Delta$-modules
\[
\HHN*{1}{G}{\pd{\idunit{L}}} \cong  \HHN*{1}{D}{\pd{\idunit{L}}}\oplus \HH*{-1}{D}{\pd{\idunit{L}}}.
\]
The following corollary gives another decomposition of $\HHN{1}{G}{\pd{\idunit{L}}}$, only as abelian groups, which turns out to be more useful than the above in the proof of Theorem \ref{yap}.

\begin{corollary}\label{cor:h1GhoGh1D}
There is an isomorphism of abelian groups
	\[
	\HHN{1}{G}{\pd{\idunit{L}}} \cong  \HHN{1}{D}{\pd{\idunit{L}}}\oplus \HH{0}{D}{\pd{\idunit{L}}}.
	\]
\end{corollary}
\begin{proof} 
The same arguments leading to \eqref{eq:shapiro} give an isomorphism
	\[
	\HHN{1}{G}{\idunit{L}}\cong\bigoplus_{\tilde{\mathfrak{l}}\subseteq \rint{F}}\HHN{1}{G(\mathfrak{L})}{\unit{L,\mathfrak{L}}}
	\] 
where for each prime $\tilde{\mathfrak{l}}$ in $F$ we have fixed a prime $\mathfrak{L}$ in $L$ above $\tilde{\mathfrak{l}}$. Therefore it is enough to show that, for each prime $\mathfrak{l}$ of $k$, 
\begin{equation}\label{eq:localcohomsplit}
\bigoplus_{\tilde{\mathfrak{l}}\mid \mathfrak{l}} \HHN{1}{G(\mathfrak{L})}{\pd{\unit{L,\mathfrak{L}}}}\cong  \HHN{1}{D(\mathfrak{L})}{\pd{\unit{L,\mathfrak{L}}}}\oplus \HH{0}{D(\mathfrak{L})}{\pd{\unit{L,\mathfrak{L}}}}\bigr).
\end{equation}
Arguing in a similar way as we did in the final part of the proof of Lemma \ref{lemma:f_in_F} we get an isomorphism
		\[
		\HHN{1}{G(\mathfrak{L})}{\unit{L,\mathfrak{L}}}\cong \Z/e_{\tilde{\mathfrak{l}}}
		\]
where $e_{\tilde{\mathfrak{l}}}$ is the ramification index of $\tilde{\mathfrak{l}}$ in $L/F$. Now if $\mathfrak{l}\in T_s(L/k)\cup T_i(L/k)$, then $e_{\tilde{\mathfrak{l}}}=e_\mathfrak{l}$, while, if $\mathfrak{l}\in T_r(L/k)$, then $e_{\tilde{\mathfrak{l}}}=e_\mathfrak{l}/2$. Thus
\begin{equation}\label{eq:loc_cohom_G}
\bigoplus_{\tilde{\mathfrak{l}}\mid \mathfrak{l}} \HHN{1}{G(\mathfrak{L})}{\pd{\unit{L,\mathfrak{L}}}}\cong\begin{cases} \Z/e_\mathfrak{l}\oplus\Z/e_\mathfrak{l}&\text{ if }\mathfrak{l}\in T_s(L/k)\\
\Z/e_\mathfrak{l}&\text{ if }\mathfrak{l}\in T_i(L/k)\\
\Z/\frac{e_\mathfrak{l}}{2}&\text{ if }\mathfrak{l}\in T_r(L/k)
\end{cases}
\end{equation}
Therefore \eqref{eq:localcohomsplit} follows by Lemma \ref{lemma:f_in_F}.
\end{proof}
	
We now come to the main result of this section. We begin with a result taken from \cite{Wal77}, which we reprove here for completeness; to state it, we introduce the notation $\dsyl{B}$ for the $2$-primary component of an abelian group $B$.

\begin{proposition}[{Walter, \cfr \cite[Theorem 5.3]{Wal77}}]\label{prop:walter}
There is an isomorphism
\[
\dsyl{\bigl(\clg{L}/\arext{L}{F}\clg{F}\bigr)}\cong \dsyl{\bigl(\clg{K}/\arext{K}{k}\clg{k}\bigr)}\oplus \dsyl{\bigl(\clg{K'}/\arext{K'}{k}\clg{k}\bigr)}.
\]
\end{proposition}
\begin{proof} 
A discussion analogous to the one following equation \eqref{eq:formula_cl_fixed} implies that the arrows $\arext{L}{F},\arext{K}{k}$ and $\arext{K'}{k}$ are injections when restricted to the $2$-primary components of class groups, because their composition with $\arnm{L}{F},\arnm{K}{k}$ and $\arnm{K'}{k}$, respectively, is multiplication by $q$. To lighten typesetting, we suppress these arrows from the notation along the proof, identifying their domain with their image.

We first claim that extending ideals to $L$ induce isomorphisms
\begin{equation}\label{eq:iso_Walter}
\dsyl{\bigl(\clg{L}^\Sigma/\clg{L}^D\bigr)}\cong \dsyl{\bigl(\clg{K}/\clg{k}\bigr)}\qquad\text{and}\qquad \dsyl{\bigl(\clg{L}^{\Sigma'}/\clg{L}^D\bigr)}\cong \dsyl{\bigl(\clg{K'}/\clg{k}\bigr)}.
\end{equation}
We prove the first isomorphism, the proof for the second being analogous. Consider the commutative diagram
\begin{equation}\begin{split}\label{diag:for_Walter}
\xymatrix{
0\ar@{->}[0,1]&\dsyl{(\clg{k})}\ar@{->}[1,0]^{\arext{L}{k}}\ar@{->}[0,1]&\dsyl{(\clg{K})}\ar@{->}[1,0]^{\arext{L}{K}}\ar@{->}[0,1]&\dsyl{\bigl(\clg{K}/\clg{k}\bigr)}\ar@{->}[1,0]\ar@{->}[0,1]&0\\
0\ar@{->}[0,1]&\dsyl{\bigl(\clg{L}^D\bigr)}\ar@{->}[0,1]&\dsyl{\bigl(\clg{L}^{\Sigma}\bigr)}\ar@{->}[0,1]&\dsyl{\bigl(\clg{L}^\Sigma/\clg{L}^D\bigr)}\ar@{->}[0,1]&0
}\end{split}\end{equation}
Thanks to Proposition \ref{servesotto} we have isomorphisms
\[
\dsyl{(\kernel{\arext{L}{k}})}\cong\kernel\Bigl(\dsyl{\HHN{1}{D}{\unit{L}}}\longrightarrow \dsyl{\HHN{1}{D}{\idunit{L}}}\Bigr)
\]
and
\[
\dsyl{(\kernel{\arext{L}{K}})}\cong\kernel\Bigl(\dsyl{\HHN{1}{\Sigma}{\unit{L}}}\longrightarrow \dsyl{\HHN{1}{\Sigma}{\idunit{L}}}\Bigr)
\]
and Lemma \ref{lemma:restr2Sylow} shows that these two kernels are isomorphic. Similarly, the isomorphisms
\[
\dsyl{(\cokernel{\arext{L}{k}})}\cong \dsyl{\HHN{1}{D}{\q{L}}}\qquad\text{and}\qquad
\dsyl{(\cokernel{\arext{L}{K}})}\cong \dsyl{\HHN{1}{\Sigma}{\q{L}}}
\]
discussed in \eqref{eq:coker_qcohom}, again combined with Lemma \ref{lemma:restr2Sylow}, show that the cokernels are isomorphic. Applying the snake lemma to the commutative diagram \eqref{diag:for_Walter}, we establish \eqref{eq:iso_Walter}.

We now apply Lemma \ref{lemma:split_Walter} to the module $B=\dsyl{(\clg{L})}$: we find an isomorphism
\[
\dsyl{\bigl(\clg{L}/\clg{L}^G\bigr)}\cong \dsyl{\bigl(\clg{L}^\Sigma/\clg{L}^D\bigr)}\oplus \dsyl{\bigl(\clg{L}^{\Sigma'}/\clg{L}^D\bigr)}
\]
which, in light of \eqref{eq:iso_Walter}, we rewrite as
\[
\dsyl{\bigl(\clg{L}/\clg{L}^G\bigr)}\cong \dsyl{\bigl(\clg{K}/\clg{k}\bigr)}\oplus \dsyl{\bigl(\clg{K'}/\clg{k}\bigr)}.
\]
We conclude the proof by observing that the map $\arext{L}{F}$ induces an isomorphism $\dsyl{(\clg{L}^G)}\cong \dsyl{(\clg{F})}$ since $2$ is coprime to $q=[L:F]$.
\end{proof}

\begin{remark}
It is easy to see that Proposition \ref{prop:walter} holds more generally for $\ell$-parts, where $\ell$ is a prime not dividing $q$. We observe moreover that it is essentially an algebraic result and very little arithmetic is needed to prove it. In fact, the class field theory we used in the proof can be avoided by extending the use of Lemma \ref{lemma:restr2Sylow}. 

Furthermore, using some integral representation theory, one can show an even more precise result: for every prime $\ell \nmid q$ there is an isomorphism
\begin{equation}\label{eq:stronger2iso}
(\clg{L})_\ell \oplus (\clg{k})_\ell^2\cong (\clg{F})_\ell \oplus (\clg{K})_\ell^2.
\end{equation}
This can be proved as follows. By looking at the corresponding characters, one first observes that there is an isomorphism of $\Q[D]$-modules
\[
\Q[D] \oplus \Q^2 \cong \Q[D/G] \oplus \Q[D/\Sigma]^2.
\] 
Thanks to Conlon's induction theorem, this implies an isomorphism of $\Z_{(\ell)}[D]$-modules
\[
\Z_{(\ell)}[D] \oplus \Z_{(\ell)}^2 \cong \Z_{(\ell)}[D/G] \oplus \Z_{(\ell)}[D/\Sigma]^2
\]
(see for instance the proof \cite[Proposition 3.9]{Ba}). The above isomorphism in turn induces \eqref{eq:stronger2iso}, since class groups form a cohomological Mackey functor (in the sense of  \cite{Bo}, see specifically \cite[Corollary 1.4]{Bo}).
\end{remark}

For a number field $M$, we let $\cln{M} =\gorder{\clg{M}}$ denote the class number of $M$.
\begin{theorem}\label{yap}
Let $L/k$ be a Galois extension of number fields whose Galois group $D=\Gal(L/k)$ is dihedral of order $2q$ with $q$ odd. Let $\Sigma\subseteq D$ be a subgroup of order $2$ and $G\subseteq D$ be the subgroup of order $q$. Set $K=L^\Sigma$ and $F=L^G$. Then
\[
\frac{\cln{L}\cln{k}^2}{\cln{F}\cln{K}^2}=\frac{\hh*{0}{D}{\pd{\unit{L}}}}{\hh*{-1}{D}{\pd{\unit{L}}}} = \frac{\hh{0}{D}{\unit{L}}}{\hh{-1}{D}{\unit{L}}}\frac{\hh{-1}{\Sigma}{\unit{L}}}{\hh{0}{\Sigma}{\unit{L}}}.
\]
\end{theorem}\begin{proof}

We start by proving the first equality. Concerning the $2$-part, this an immediate consequence of Proposition \ref{prop:walter}, observing that the cohomology groups with values in the uniquely $2$-divisible module $\pd{\unit{L}}$ have trivial $2$-components.

We now focus on the odd part. From Proposition \ref{loddpart}, we have
\begin{equation}\label{eq:startingpoint}
\gorder*{\clgp{L}}=\gorder*{\clgp{K}}^2\frac{\gorder*{\clgp{L}^{G}}}{\hhN*{0}{D}{\clgp{L}}\hh*{-1}{D}{\clgp{L}}  \gorder*{N_D \clgp{L}}}.
\end{equation}
Note that Proposition \ref{servesotto} with $H=G$ gives 
\begin{equation}\label{eq:ambclassew}
\gorder*{\clgp{L}^{G}} = \gorder*{\clgp{F}} \cdot \frac{ \hhN*{1}{G}{\pd{\idunit{L}}}\hhN*{1}{G}{\pd{\princ{L}}}}{ \hhN*{1}{G}{\pd{\unit{L}}}},
\end{equation}
and we can also replace the term $\hhN{1}{G}{\pd{\princ{L}}}$ by $\hhN{1}{D}{\pd{\princ{L}}}\hh{-1}{D}{\pd{\princ{L}}}$ in view of Proposition \ref{fixcohom}. By Corollary \ref{cor:h1GhoGh1D} we have 
\begin{equation}\label{eq:idunitDG}
\hhN*{1}{G}{\pd{\idunit{L}}}=\hhN*{1}{D}{\pd{\idunit{L}}}\hh*{0}{D}{\pd{\idunit{L}}}.
\end{equation}
Thus \eqref{eq:startingpoint} becomes 
\begin{align*}
\frac{\gorder*{\clgp{L}}}{\gorder*{\clgp{K}}^2}  &= \frac{\gorder*{\clgp{L}^G}}{\hhN*{0}{D}{\clgp{L}}\hh*{-1}{D}{\clgp{L}}\cdot \gorder*{N_D \clgp{L}}}\\
\text{(by \eqref{eq:ambclassew})} &=\gorder*{\clgp{F}}\cdot \frac{\hhN*{1}{G}{\pd{\idunit{L}}}\hhN*{1}{D}{\pd{\princ{L}}}\hh*{-1}{D}{\pd{\princ{L}}}}{\hhN*{1}{G}{\unit{L}}\hhN*{0}{D}{\clgp{L}} \hh*{-1}{D}{\clgp{L}} \cdot \gorder*{N_D \clgp{L}}}\\
\text{(by \eqref{eq:idunitDG})}&= \gorder*{\clgp{F}} \cdot \frac{\hhN*{1}{D}{\pd{\idunit{L}}}\hhN*{1}{D}{\pd{\princ{L}}}}{\hhN*{1}{G}{\pd{\unit{L}}}\hhN*{0}{D}{\clgp{L}} }\cdot\frac{\hh*{0}{D}{\pd{\idunit{L}}}\hh*{-1}{D}{\pd{\princ{L}}}}{\hh*{-1}{D}{\clgp{L}} \cdot \gorder*{N_D \clgp{L}}}\\
\text{(by Proposition \ref{prop:gessisullalavagna})}&=\gorder*{\clgp{F}}\cdot \frac{\hhN*{1}{D}{\pd{\idunit{L}}}\hhN*{1}{D}{\pd{\princ{L}}}}{\hhN*{1}{G}{\pd{\unit{L}}}\hhN*{0}{D}{\clgp{L}} }\cdot\frac{\hh*{0}{D}{\pd{\unit{L}}}}{\gorder*{\clgp{k}}}\\
\text{(by Proposition \ref{servesotto} with $H=D$)}&=\frac{\gorder*{\clgp{F}}}{\hhN*{1}{G}{\pd{\unit{L}}}}\cdot\frac{\hhN*{1}{D}{\pd{\unit{L}}}}{\gorder*{\clgp{k}} }\cdot\frac{\hh*{0}{D}{\pd{\unit{L}}}}{\gorder*{\clgp{k}}}\\
\text{(by Proposition \ref{fixcohom})} &= \frac{\gorder*{\clgp{F}}}{\gorder*{\clgp{k}}^2}\cdot\frac{\hh*{0}{D}{\pd{\unit{L}}}}{\hh*{-1}{D}{\pd{\unit{L}}}}.
\end{align*}
This completes the proof of the first equality of the theorem. 
		
To obtain the second equality in the statement, observe that 
\[
\HH*{i}{D}{\pd{\unit{L}}}\cong \HH{i}{D}{\unit{L}}/\dsyl{\HH{i}{D}{\unit{L}}}
\]
where $\dsyl{\HH{i}{D}{\unit{L}}}$ is the $2$-primary component of $\HH{i}{D}{\unit{L}}$. We conclude by Lemma \ref{lemma:restr2Sylow}.
\end{proof}   
	
Now we recall the definition of the Herbrand quotient of a $G$-module $B$ as
\[
\herbr{G}{B}=\frac{\hh{0}{G}{B}}{\hhN{1}{G}{B}}
\]
whenever the quotient is well-defined. The following result will be crucial for our computations.

\begin{herbrand}[see {\cite[Corollary 2 of Theorem 1, Chapter IX]{La}}] The Herbrand quotient of the $G$-module $\unit{L}$ is
\[
\herbr{G}{\unit{L}}=\frac{1}{q}.
\]
\end{herbrand}
	
Using the value of the Herbrand quotient of units, we can easily obtain bounds on the value of the ratio of class numbers of subfields of $L$. For a number field $M$ we set
\[
\defect{M}{q}=
\begin{cases}
0&\text{if $\roots[q]{M}$ is trivial,}\\
1&\text{otherwise}
\end{cases}
\]   	
where $\roots{M}$ denotes the group of roots of unity of $M$and $\roots[n]{M}$ is the subgroup of roots of unity of $M$ killed by $n\geq 1$. In what follows, for a prime number $\ell$, $v_\ell$ denotes the $\ell$-adic valuation.

\begin{corollary} \label{cor:bounds} 
Let notation be as in Theorem \ref{yap}. For every prime $\ell$, the following bounds hold
\[
-av_\ell(q)\leq v_\ell\left(\frac{\cln{L}\cln{k}^2}{\cln{F}\cln{K}^2}\right)\leq bv_\ell(q) 
\]
where $a=\rank_\Z(\unit{F}) + \defect{F}{q}+1$ and $b=\rank_\Z(\unit{k})+\defect{k}{q}$. In particular, if $F$ is totally real of degree $[F\colon\Q]=2d$, then
\[
-2dv_\ell(q)\leq v_\ell\left(\frac{\cln{L}\cln{k}^2}{\cln{F}\cln{K}^2}\right)\leq (d-1)v_\ell(q). 
\]
\end{corollary}
\begin{proof}
Thanks to Theorem \ref{yap}, we only have to prove that the prescribed bounds hold for the ratio
\[
\hh*{0}{D}{\pd{\unit{L}}}/\hh*{-1}{D}{\pd{\unit{L}}}.
\]
Observe that $\HH{0}{G}{\pd{\unit{L}}}=\pd{\unit{F}}/N_G\pd{\unit{L}}$ is a quotient of $\pd{\unit{F}}$ annihilated by $q$. In particular its order divides $q^{\rank_\Z(\unit{F})+\defect{F}{q}}$. Using the value of the Herbrand quotient of units, we therefore have 
\begin{align*}
v_\ell\bigl(\hh*{-1}{D}{\pd{\unit{L}}}\bigr) &\leq v_\ell\bigl(\hh*{-1}{G}{\pd{\unit{L}}}\bigr) = v_\ell(q)+v_\ell\bigl(\hh{0}{G}{\pd{\unit{L}}}\bigr)
\\ &\leq v_\ell(q) + \bigl(\defect{F}{q} + \rank_\Z(\unit{F})\bigr)v_\ell(q)\\
&=av_\ell(q).
\end{align*}
We deduce that
\[
v_\ell\left(\frac{\hh*{0}{D}{\pd{\unit{L}}}}{\hh*{-1}{D}{\pd{\unit{L}}}}\right)=v_\ell\bigl(\hh*{0}{D}{\pd{\unit{L}}}\bigr)- v_\ell\bigl(\hh*{-1}{D}{\pd{\unit{L}}}\bigr) \geq- v_\ell\bigl(\hh*{-1}{D}{\pd{\unit{L}}}\bigr)\geq -av_\ell(q).
\]
Similarly, $\HH{0}{D}{\pd{\unit{L}}}=\pd{\unit{k}}/N_D\pd{\unit{L}}$ is a quotient of $\pd{\unit{k}}$ annihilated by $q$. In particular its order divides $q^{\rank_\Z(\unit{k})+\defect{k}{q}}$ and therefore
\[
v_\ell\bigl(\hh*{0}{D}{\pd{\unit{L}}}\bigr)\leq v_\ell(q^{\rank_\Z(\unit{k})+\defect{k}{q}}) = bv_\ell(q). 
\]
We deduce that 
\[	v_\ell\left(\frac{\hh*{0}{D}{\pd{\unit{L}}}}{\hh*{-1}{D}{\pd{\unit{L}}}}\right)=v_\ell\bigl(\hh*{0}{D}{\pd{\unit{L}}}\bigr)- v_\ell\bigl(\hh*{-1}{D}{\pd{\unit{L}}}\bigr)  \leq v_\ell\bigl(\hh*{0}{D}{\pd{\unit{L}}}\bigr) \leq bv_\ell(q).
\]

When $F$ is totally real, it contains no roots of unity different from $\pm 1$ so that $a=2d$ and $b=(d-1)$.
\end{proof}

\begin{remark}
It would be interesting to understand to what extent the bounds of Corollary \ref{cor:bounds} are optimal. To be more precise, for given $a, b\in \Z$, consider the set 
\[
S(a,b) = \textrm{$\{c\in \Q$ such that $-av_\ell(q) \leq v_\ell(c) \leq bv_\ell(q),$ for every prime $\ell\}$}.
\] 
One  could ask whether, for every $a \geq 1$ and $b\geq 0$ and every $c\in S(a,b)$, there exist a number field $k_c$, a quadratic extension $F_c/k_c$ satisfying
\[
a = \rank_\Z(\unit{F_c})+\defect{F_c}{q}+1\qquad\text{and}\qquad b = \rank_\Z(\unit{k_c}) + \defect{k_c}{q}
\]
and a dihedral extensions $L_c/k_c$ of degree $2q$ such that $F_c\subseteq L_c$ and 
\[
\frac{\cln{L_{c}}\cln{k_c}^2}{\cln{F_c}\cln{K_c}^2} = c,
\]
where $K_c/k_c$ is a subextension of $L_c/k_c$ of degree $q$.
\end{remark}
		
The bounds of Corollary \ref{cor:bounds} hold in full generality, \ie for an arbitrary base field $k$. When $k$ is totally real and $F$ is a CM field, these bounds can be improved, as we show in Corollary \ref{cor:boundsCM}, which generalizes \cite[Example 6.3]{Ba}. We start with the following lemma, whose first statement  is well-known.

\begin{lemma}\label{lemma:rootsunity}
We have
\[
\roots{F}=\roots{L}\qquad\text{and}\qquad\roots{k}=\roots{K}.
\]
Moreover, the intersection $I_G\unit{L}\cap \unit{F}$ is a finite cyclic group and there is a $\Delta$-antiequivariant isomorphism
\[
\bigl((\unit{L})^q\cap \unit{F}\bigr)/(\unit{F})^q \cong I_G\unit{L}\cap \unit{F}.
\]
In particular we have isomorphisms of abelian groups
\[
\Bigl((\unit{L})^q\cap\unit{F}/(\unit{F})^q\Bigr)^\mp \cong  \Bigl(I_G\unit{L}\cap \unit{F}\Bigr)^\pm.
\]
\end{lemma}
\begin{proof}
Let $\zeta$ be a root of unity of $L$, and set $M=k(\zeta)$. Then $M/k$ is an abelian extension of $k$ contained in $L$ and therefore $M\subseteq F$. In particular $\roots{F}=\roots{L}$ and, taking $\Sigma$-invariants, $\roots{k}=\roots{K}${, so that the first statement is proved}.

To prove the rest of the lemma, consider the Kummer sequence
\begin{equation}\label{eq:kummer}
1\longrightarrow\roots[q]{L}\longrightarrow\unit{L}\overset{(\cdot)^q}{\longrightarrow}(\unit{L})^q\longrightarrow 1.
\end{equation}
Taking $G$-cohomology, we obtain a $\Delta$-equivariant isomorphism
\begin{equation}\label{eq:delta_iso_schifo=ker}
\bigl((\unit{L})^q\cap \unit{F}\bigr)/(\unit{F})^q=\kernel\bigl(\HHN{1}{G}{\roots[q]{L}}\to\HHN{1}{G}{\unit{L}}\bigr).
\end{equation}
Using Proposition \ref{fixcohom} we deduce a $\Delta$-antiequivariant isomorphism
\[
\bigl((\unit{L})^q\cap \unit{F}\bigr)/(\unit{F})^q\cong\kernel\bigl(\HH{-1}{G}{\roots[q]{L}}\to\HH{-1}{G}{\unit{L}}\bigr).
\] 
By the first statement of the lemma, the $D$-module $B=\pd{\unit{L}}$ satisfies $\tors{B}^G=\tors{B}$ and we can apply Lemma \ref{lemma:schifo=nucleo}. In particular we obtain a $\Delta$-antiequivariant isomorphism 
\[
\bigl((\unit{L})^q\cap \unit{F}\bigr)/(\unit{F})^q\cong I_G\unit{L}\cap \unit{F}.
\] 
Observe that, again by Lemma \ref{lemma:schifo=nucleo}, $I_G\unit{L}\cap \unit{F}$ is equal to $I_G\unit{L}\cap \roots[q]{F}\subseteq \roots{F}$, whence the claim that it is cyclic. 
\end{proof}

We can now state the following result which sharpens Corollary \ref{cor:bounds} under the assumption that $F$ is a CM field and $F^+=k$. We write $\roots[{q^\infty}]{F}$ to denote the roots of unity $\zeta\in \roots{F}$ such that $\zeta^{q^m}=1$ for some $m\geq 0$: it is clearly a finite group.

\begin{corollary}\label{cor:boundsCM}
With notation as in Theorem \ref{yap}, assume moreover that $F$ is a CM field of degree $[F:\Q]=2d$ and that $k$ is its totally real subfield. Let $s$ be the order of the quotient $\bigl((\unit{K})^q\cap\unit{k}\bigr)/(\unit{k})^q$ and let $t$ be the order of $\roots[{q^\infty}]{F}/\roots[q]{F}$. Then, for every prime $\ell$,
\[
-v_\ell(q)-v_\ell(s)-v_\ell(t)\leq v_\ell\left(\frac{\cln{L}\cln{k}^2}{\cln{F}\cln{K}^2}\right)\leq (d-1)v_\ell(q).
\]
\end{corollary}
\begin{proof}
Again, it suffices to study the ratio $\hh{0}{D}{\pd{\unit{L}}}/\hh{-1}{D}{\pd{\unit{L}}}$, thanks to Theorem \ref{yap}. For $i=-1,0$ we write $\hh{i}{D}{\pd{\unit{L}}}=\hh{i}{G}{\unit{L}}/\hh{i}{G}{\unit{L}}^-$, where we denote by $\hh{i}{G}{\unit{L}}^-$ the order of the minus-part of $\HH{i}{G}{\unit{L}}$ (observe that $\HH{i}{G}{\unit{L}}=\HH{i}{G}{\pd{\unit{L}}}$ since $G$ is of odd order). We find
\begin{equation}\label{eq:spezzatino}
\frac{\hh*{0}{D}{\pd{\unit{L}}}}{\hh*{-1}{D}{\pd{\unit{L}}}}=\frac{\hh{0}{G}{\unit{L}}}{\hh{-1}{G}{\unit{L}}}\cdot\frac{\hh{-1}{G}{\unit{L}}^-}{\hh{0}{G}{\unit{L}}^-}=\frac{1}{q}\cdot\frac{\hh{-1}{G}{\unit{L}}^-}{\hh{0}{G}{\unit{L}}^-}.
\end{equation}

We now use that $F$ is a CM field and $F=k^+$. Starting from the tautological exact sequence defining $\HH{0}{G}{\unit{L}}$, we can consider its minus-part to obtain
\[
0\longrightarrow \bigl(\arnm{L}{F}\pd{\unit{L}}\bigr)^-\longrightarrow\bigl(\pd{\unit{F}}\bigr)^-\longrightarrow \HH{0}{G}{\unit{L}}^-\longrightarrow 0.
\]
The quotient $\unit{F}/\unit{k}\roots{F}$ is of order $1$ or $2$ by  \cite[Chapter XIII, \S2, Lemma~1]{LaCF} and, in particular, $\pd{\unit{F}}^-=\pd{\roots{F}}$. Since the rightmost term of the above sequence is killed by $q$, the sequence can be rewritten as
\begin{equation}\label{seq:mu_surjects}
0\longrightarrow \bigl(\pd{\arnm{L}{F}\unit{L}}\bigr)^-\longrightarrow \roots[{q^\infty}]{F}\longrightarrow \HH{0}{G}{\unit{L}}^-\longrightarrow 0.
\end{equation}
On the other hand, thanks to Lemma \ref{lemma:rootsunity}, the $G$-action on $\roots{L}$ is trivial and (as in the proof of Lemma \ref{lemma:schifo=nucleo}) we find $\roots[q]{F}=\HH{-1}{G}{\roots[q]{L}}=\HH{0}{G}{\roots[q]{L}}$. By considering the Kummer sequence \eqref{eq:kummer} and taking the minus part of its $G$-cohomology we obtain
\begin{equation}\label{diag:long_G_kummer}
\begin{aligned}
\xymatrix{
0\ar@{->}[0,1]&\ar@{->}[0,1]I_G\unit{L}\cap\roots[q]{F}&\ar@{->}[0,1]\HH*{-1}{G}{\roots[q]{L}}^-=\roots[q]{F}\ar@{->}[0,1]&\HH{-1}{G}{\unit{L}}^-
\ar `r[d] `[d] `^l[dlll] `[ddlll] `[ddll] [ddll]&\\
&&&&\\
&\HH*{-1}{G}{(\unit{L})^q}^-\ar@{->}[0,1]&\roots[q]{F}=\HH{0}{G}{\roots[q]{L}}^-
\ar@{->}[0,1]&Y\ar@{->}[0,1]&0
}
\end{aligned}\end{equation}
where $Y\subseteq \HH{0}{G}{\unit{L}}^-$ is defined by the exactness of the diagram. Taking orders in the above long exact sequence we deduce
\[
\hh{-1}{G}{\unit{L}}^-=\frac{\hh{-1}{G}{(\unit{L})^q}^-\cdot \gorder{Y}}{\gorder{I_G\unit{L}\cap\roots[q]{F}}}
\]
while \eqref{seq:mu_surjects} implies that $\hh{0}{G}{\unit{L}}^-\leq \gorder{\roots[{q^\infty}]{F}}$. Therefore \eqref{eq:spezzatino} becomes
\begin{equation*}
\frac{\hh*{0}{D}{\pd{\unit{L}}}}{\hh*{-1}{D}{\pd{\unit{L}}}}\geq\frac{1}{q}\cdot\frac{\hh{-1}{G}{(\unit{L})^q}^-\cdot\gorder{Y}}{\gorder{I_G\unit{L}\cap\roots[q]{F}}}\cdot\frac{1}{\gorder{\roots[{q^{\infty}}]{F}}}.
\end{equation*}
The second line of the long exact sequence \eqref{diag:long_G_kummer} shows that $\hh{-1}{G}{(\unit{L})^q}^-\cdot\gorder{Y}\geq \gorder{\roots[q]{F}}$ so that we obtain
\begin{equation}\label{eq:qquadro}
\frac{\hh*{0}{D}{\pd{\unit{L}}}}{\hh*{-1}{D}{\pd{\unit{L}}}}\geq\frac{1}{q}\cdot\frac{1}{\gorder{I_G\unit{L}\cap\roots[q]{F}}}\cdot\frac{1}{t}.
\end{equation}
Since $\pd{\roots{F}}=\pd{\unit{F}}^-$, we have $I_G\unit{L}\cap\roots[q]{F}=\bigl(I_G\unit{L}\cap{\roots{F}}\bigr)^-=\bigl(I_G\unit{L}\cap\unit{F}\bigr)^-$, by Lemma \ref{lemma:schifo=nucleo}, and applying Lemma \ref{lemma:rootsunity} we deduce that $I_G\unit{L}\cap\roots[q]{F}$ is of order $s$ because
\[
\Bigl(\bigl(\unit{L}\bigr)^q\cap\unit{F}/\bigl(\unit{F}\bigr)^q\Bigr)^+=\bigl((\unit{K})^q\cap\unit{k}\bigr)/\bigl(\unit{k}\bigr)^q.
\]
Taking $\ell$-adic valuations in \eqref{eq:qquadro}, we obtain the bound
\[
-v_\ell(q)-v_\ell(s)-v_\ell(t)\leq v_\ell\left(\frac{\cln{L}\cln{k}^2}{\cln{F}\cln{K}^2}\right).
\]
As for the other inequality of the corollary, we argue as in Corollary \ref{cor:bounds} observing that since $k$ is totally real, $\defect{k}{q}=1$ and $\rank_{\Z}(\unit{k})=d-1$.
\end{proof}

\begin{remark}\label{rmk:schifo_per_bartels}
The quotient $\bigl((\unit{K})^q\cap \unit{k}\bigr)/(\unit{k})^q$ has exponent dividing $q$ and is cyclic (by Lemma \ref{lemma:rootsunity}), hence its order $s$ divides $q$. We claim that $s=q$ if and only if $K=k(\sqrt[q]{u})$ for some unit $u\in\unit{k}$. This is equivalent to the fact that, given $u\in (\unit{K})^q\cap \unit{k}$, the class $[u]\in \bigl((\unit{K})^q\cap \unit{k}\bigr)/(\unit{k})^q$ has order strictly less than $q$ if and only if $k(x)\subsetneq K$, where $x\in\unit{K}$ is such that $x^q=u$. Note that $k(x)\subsetneq K$ precisely when $X^q-u$ is reducible, a condition equivalent to the existence of a prime $p\mid q$ such that $u \in (\unit{k})^p$ (see \cite[Chapter~ VI, Theorem~9.1]{Lan02}). We are thus reduced to show that the class $[u]\in \bigl((\unit{K})^q\cap \unit{k}\bigr)/(\unit{k})^q$ has order strictly less than $q$ if and only if there is $p\mid q$ such that $u \in (\unit{k})^p$.

If $u\in(\unit{k})^p$, then clearly $u^{q/p}\in (\unit{k})^q$ and the order of $[u]$ strictly divides $q$. Suppose, conversely, that $[u]$ has order $q/p$ for some prime $p$, so $u^{q/p}\in (\unit{k})^q$, and let $v\in \unit{k}$ be such that $u^{q/p}=v^q$. We deduce that $v^q = (x^{q/p})^q$ or, equivalently, $x^{q/p}=v\zeta$ for some $\zeta\in \roots{K}$. Hence $x^{q/p}\in \unit{k}$, by Lemma \ref{lemma:rootsunity}, and therefore $u=x^q=(x^{q/p})^p\in (\unit{k})^p$. This shows our claim.

When $q=p$ is a prime and $\roots[{p^\infty}]{F}=\roots[p]{F}$, we can restate Corollary \ref{cor:boundsCM} by saying that the ratio of class numbers equals $p^{m}$ for some $-2\leq m\leq (d-1)$ or $-1\leq m\leq (d-1)$ according as whether $K=k(\sqrt[p]{u})$ for some $u\in\unit{k}$, or not.
\end{remark}

The next corollary (which generalizes \cite[\S 4]{HK} and \cite[Th\'eor\`eme IV.1]{Mos79}) shows that if $k=\Q$ even sharper bounds hold.

\begin{corollary}\label{cor:boundsQ} 
Let $L/\Q$ be a dihedral extension of degree $2q$ where $q$ is odd, and let $F/\Q$ be the quadratic extension contained in $L$. Denote by $K$ the field fixed by a subgroup of order $2$ in $\Gal(L/\Q)$. Then, for every prime $\ell$,
\[
0 \geq v_\ell\left(\frac{\cln{L}}{\cln{F}\cln{K}^2}\right)\geq
\begin{cases}
-2v_\ell(q)&\text{if $F$ is real quadratic}\\
-v_\ell(q)&\text{if $F$ is imaginary quadratic.}
\end{cases}
\] 
\end{corollary}
\begin{proof}
The case when $F$ is real quadratic is already contained in Corollary \ref{cor:bounds}. If $F$ is totally imaginary, it is a CM field so that we can apply Corollary \ref{cor:boundsCM}: we need to prove that $s=t=1$. Recall that $s$ was defined as the order of the quotient
\[
\bigl((\unit{K})^q\cap\{\pm 1\}\bigr)/\{\pm 1\},
\]
which is trivial. As for $t$, the only case in which $\roots[{q^\infty}]{F}\neq 1$ occurs when $F=\Q(\sqrt{-3})$ and $3\mid q$, but in this case $\roots[{q^\infty}]{F}=\roots[{q}]{F}$, so $t=1$.
\end{proof}

\section{Class number formula in terms of a unit index}\label{sec:equivalence_unit_index}
In this section we rewrite the formula of Theorem \ref{yap} in a form more similar to the one often found in the literature, namely replacing the orders of cohomology groups by the index of a subgroup of the group of units of $L$. 	
\begin{notation}
Recall that for an abelian group $B$ and a natural number $n$, we write $\tors[n]{B}=\{b\in B : nb=0\}$. We set
\[
\tor_\Z(B) = \bigcup_{n\in\N}\tors[n]{B}
\]
for the torsion subgroup of $B$ and $\overline{B} = B/\tor_\Z(B)$ for the maximal torsion-free quotient of $B$. 
\end{notation}
		
\begin{proposition}\label{prop:unitindexcohom}	
When $B$ is a $D$-module which is $\Z$-finitely generated the following formula holds:
\[
\frac{\hh*{0}{D}{\pd{B}}}{\hh*{-1}{D}{\pd{B}}}=\gindex{B}{B^{\Sigma}+B^{\Sigma'}+B^G}\herbr{G}{B}\frac{\gorder*{I_GB\cap B^D}}{\gorder*{I_GB\cap B^G}}q^{\rank_\Z(B^D)-\rank_\Z(B^G)}.
\]
\end{proposition}

Proposition \ref{prop:unitindexcohom}, as well as its proof, is purely algebraic and requires no arithmetic. We postpone its proof after Theorem \ref{thm:di_Bartel} and focus here on its arithmetic application.

\begin{corollary}\label{cor:formulawithunitindex}
With notation as in Theorem \ref{yap}, one has
\[
\frac{\cln{L}\cln{k}^2}{\cln{F}\cln{K}^2}=\gindex*{\unit{L}}{\unit{K}\unit{K'}\unit{F}}\frac{q^{(\rank_\Z(\unit{k})-\rank_\Z(\unit{F}) -1)}}{\gindex*{(\unit{K})^q\cap \unit{k}}{(\unit{k})^q}}.
\]
and the term $\gindex{(\unit{K})^q\cap \unit{k}}{(\unit{k})^q}$ divides $q$.
\end{corollary}
\begin{proof} 
The formula follows by combining Theorem \ref{yap}, Proposition \ref{prop:unitindexcohom}, the value of the Herbrand quotient of units and Lemma \ref{lemma:rootsunity}. 
\end{proof}

The formula of Corollary \ref{cor:formulawithunitindex} is similar to those appeared in the literature (see the Introduction for references). Sometimes the right-hand side is expressed differently but this is just a matter of rearranging terms. For instance, we prove below that Bartel's formula in \cite {Ba} coincides with ours when $q$ is a prime.
	
\begin{theorem}[Bartel, see~{\cite[Theorem~1.1]{Ba}}] \label{thm:di_Bartel}
Let $L/k$ be a Galois extensions of number fields with Galois group $D_{2p}$ for $p$ an odd prime, let $F$ be the intermediate quadratic extension and $K,K'$ be distinct intermediate extensions of degree $p$. Set $\delta$ to be $3$ if $K/k$ is obtained by adjoining the $p$th root of a non-torsion unit (thus so is $L/F$) and $1$ otherwise. Then we have
\[
\frac{\cln{L}\cln{k}^2}{\cln{F}\cln{K}^2}=p^{\alpha/2}\cdot\gindex*{\unit{L}}{\unit{K}\unit{K'}\unit{F}}
\]
where $\alpha=2\rank_\Z(\unit{k})-\rank_\Z(\unit{F}) - \frac{\rank_\Z(\unit{L})-\rank_\Z(\unit{F})}{p-1}-\delta$.
\end{theorem}

\begin{remark} 
Observe that Bartel's theorem assumes that $q=p$ is an odd prime. However Bartel's formulation is more general than the one given above, since he obtains a formula for general $S$-class numbers where $S$ is any finite set of primes in $k$ containing the archimedean ones. Since the proof of Theorem \ref{yap} ultimately relies on the study of $G$-cohomology of Diagram \ref{diag:magic} and a $S$-version of the diagram exists, our main result should easily generalize to $S$-class numbers and $S$-units.
\end{remark}

\begin{proof} 
Observe that since $L/F$ is Galois of odd order, we have $\rank_\Z(\unit{L}) = p\rank_\Z(\unit{F}) +(p-1)$ and so
\[
\frac{\rank_\Z(\unit{L})-\rank_\Z(\unit{F})}{p-1}=\frac{(p-1)(\rank_\Z(\unit{F})+1)}{p-1}=\rank_\Z(\unit{F})+1.
\]
Therefore
\[
\alpha=2\rank_\Z(\unit{k})-2\rank_\Z(\unit{F})-1-\delta.
\]
Using this expression for $\alpha$ and comparing Corollary \ref{cor:formulawithunitindex} with the statement of the theorem, we simply need to show that 
\[
p^{\frac{\delta-1}{2}}=\gindex*{(\unit{K})^{p}\cap \unit{k}}{(\unit{k})^{p}}.
\]
By definition of $\delta$, the above reduces, using that $\bigl((\unit{K})^p\cap \unit{k}\bigr)/(\unit{k})^p$ is cyclic by Lemma \ref{lemma:rootsunity}, to prove that
\[
\gindex*{(\unit{K})^p\cap \unit{k}}{(\unit{k})^p}\neq 0 \Longleftrightarrow K=k(\sqrt[p]{u})\quad\text{ for some }u\in\unit{k}
\]
and this is clear (see also Remark \ref{rmk:schifo_per_bartels}).
\end{proof}

We now prove Proposition \ref{prop:unitindexcohom}. The strategy follows very closely the one of \cite[Section 2]{CapBra}, which, in turn, is partly inspired by \cite[Section 5]{Le}. 

\begin{lemma}\label{lemma:unitindexppower}
Let $B$ be a $D$-module which is $\Z$-finitely generated. Then the index of $B^{\Sigma}+B^{\Sigma'}+B^{G}$ in $B$ is finite and
\[
\gindex{B}{B^{\Sigma}+B^{\Sigma'}+B^{G}} = \gindex*{N_GB}{N_G(B^{\Sigma}+B^{\Sigma'}+B^{G})}\cdot\gindex*{B[N_G]}{B[N_G]\cap (B^{\Sigma} + B^{\Sigma'} + B^{G})}.
\]
Moreover, the three terms appearing in the above equality are coprime to any prime not diving $q$. 
\end{lemma}
\begin{proof}
By equation \eqref{eq:snake} with $f=N_G$ and $B'=B^{\Sigma}+B^{\Sigma'}+B^G$, we only need to verify that the above indices are finite. Since $B$ is $\Z$-finitely generated, all $G$-Tate cohomology groups of $B$ are finite of exponent dividing $q$. In particular, $B[N_G]\cap \bigl(B^{\Sigma} + B^{\Sigma'} + B^{G}\bigr)$ is of finite index in $B[N_G]$ since it contains $I_GB$ by Lemma \ref{lemma:2incl}. This argument also shows that $\gindex{B[N_G]}{B[N_G]\cap (B^{\Sigma} + B^{\Sigma'} + B^{G})}$ is coprime to any prime not dividing $q$. On the other hand, $N_G\bigl(B^{\Sigma}+B^{\Sigma'}+B^{G}\bigr)$ contains $q(B^{G})$, so its index in $B^G$ is finite and coprime to any prime not dividing $q$. The same holds therefore for the index of $N_G\bigl(B^{\Sigma}+B^{\Sigma'}+B^{G}\bigr)$ in $N_GB$, since $N_GB\subseteq B^G$. This finishes the proof.
\end{proof}

We now analyze separately the two factors of the right-hand side of the equality of Lemma \ref{lemma:unitindexppower}.
\begin{lemma}\label{lemma:firstfactor}
Let $B$ be a $D$-module which is $\Z$-finitely generated. Then the following equality holds:
\[
\gindex*{N_GB}{N_G(B^{\Sigma}+B^{\Sigma'}+B^{G})} = \frac{\hh*{0}{D}{\pd{B}}\gindex*{B^G}{q(B^{G})}}{\hh*{0}{G}{\pd{B}}\gindex*{B^D}{q(B^{D}\vphantom{B^G})}}.
\]
\end{lemma}
\begin{proof}
Every factor appearing in the formula of the statement is coprime to any prime not dividing $q$ by Lemma \ref{lemma:unitindexppower}. This means that, replacing $B$ by $\pd{B}$, we can assume that $B$ is uniquely $2$-divisible (note that Tate cohomology groups remain finite) so that $\Sigma$-Tate cohomology of $B$ is trivial, and similarly for $\Sigma'$. In particular 
\[
N_G(B^\Sigma)=N_G(N_\Sigma B)=N_DB
\]
and similarly $N_G(B^{\Sigma'})=N_G(N_{\Sigma'}B)=N_DB$. Therefore
\begin{align*}
\gindex*{N_GB}{N_G(B^{\Sigma}+B^{\Sigma'}+B^{G})}& =\gindex{N_GB}{q(B^G)+N_DB} \\
&=\frac{\gindex*{B^G}{q(B^G)+N_DB}}{\hh{0}{G}{B}}\\
&=\frac{\gindex*{B^G}{q(B^G)}}{\gindex*{q(B^G) + N_DB\vphantom{B^G}}{q(B^G)\vphantom{B^{G}}}\cdot\hh{0}{G}{B}}\\
&=\frac{\gindex*{B^G}{qB^G}}{\gindex{N_DB}{qB^G\cap N_DB}\cdot\hh{0}{G}{B}}\\
\text{(see \eqref{eq:normqpow} below)}&=\frac{\gindex*{B^G}{q(B^G)}}{\gindex*{N_DB}{q(B^D)\vphantom{B^G}}\cdot\hh{0}{G}{B}}\\
&=\frac{\gindex*{B^G}{q(B^G)}\cdot\hh{0}{D}{B}}{\gindex*{B^D}{q(B^D\vphantom{B^G})}\cdot\hh{0}{G}{B}}.
\end{align*}
We have used that
\begin{equation}\label{eq:normqpow}
q(B^G)\cap N_D B = q(B^D). 
\end{equation}
To see this, first observe that
$q(B^G)\cap B^D = q(B^D)$
because the surjective map $B^G\stackrel{q}{\longrightarrow}q(B^G)$ stays surjective after taking $\Sigma$-invariants, since $B^G$ is uniquely $2$-divisible. In particular, we have $q(B^G)\cap N_DB\subseteq q(B^D)$ which implies
\[
q(B^G)\cap N_DB = q(B^D)\cap N_DB.
\]
Now \eqref{eq:normqpow} follows since
\[
q(B^D)\cap N_DB = q(B^D)
\]
because
\[
q(B^D) =N_G(B^D) = N_G(N_\Sigma (B^G) ) = N_D(B^G) \subseteq N_DB. \qedhere
\]
\end{proof}

\begin{lemma}\label{lemma:secondfactor}
Let $B$ be a $D$-module which is $\Z$-finitely generated. Then
\[
\gindex*{B[N_G]}{B[N_G]\cap\bigl(B^{\Sigma}+B^{\Sigma'}+B^{G}\bigr)} = \frac{\hhN*{1}{G}{\pd{B}}\cdot \gindex*{I_GB \cap \tors{B^G}}{I_GB\cap \tors{B^D}}}{\hh*{-1}{D}{\pd{B}}\cdot \gindex*{\tors{B^G}}{\tors{B^D}}}.
\]
\end{lemma}
\begin{proof}
Observe that $\bigl(B[N_G]:B[N_G]\cap\bigl(B^{\Sigma}+B^{\Sigma'}+B^{G}\bigr)\bigr)$ is coprime to any prime not dividing $q$ (by Lemma \ref{lemma:unitindexppower}). The same holds for all the factors of the right-hand side of the equality of the lemma. This means that, replacing $B$ by $\pd{B}$, we can assume that $B$ is uniquely $2$-divisible.

Observe first that 
\begin{align*}
\gindex*{B[N_G]}{B[N_G]\cap\bigl(B^{\Sigma}+B^{\Sigma'}+B^{G}\bigr)} &= \frac{\gindex*{B[N_G]}{B[N_G]\cap\bigl(B^{\Sigma}+B^{\Sigma'}\bigr)}}{\gindex{B[N_G]\cap\bigl(B^{\Sigma}+B^{\Sigma'}+B^{G}\bigr)}{B[N_G]\cap\bigl(B^{\Sigma}+B^{\Sigma'}\bigr)}}\\
\text{(by Lemma \ref{lemma:quozcoom})}&= \frac{\hhN{1}{D}{B}}{\gindex*{B[N_G]\cap\bigl(B^{\Sigma}+B^{\Sigma'}+B^{G}\bigr)}{B[N_G]\cap\bigl(B^{\Sigma}+B^{\Sigma'}\bigr)}}
\end{align*}
Since $\hhN{1}{D}{B} = \hh{-1}{G}{B}/\hh{-1}{D}{B}$ by Proposition \ref{fixcohom}, we only need to prove that 
\[
\gindex*{B[N_G]\cap\bigl(B^{\Sigma}+B^{\Sigma'}+B^{G}\bigr)}{B[N_G]\cap\bigl(B^{\Sigma}+B^{\Sigma'}\bigr)} = \frac{\gindex*{\tors{B^G}}{\tors{B^D}}}{\gindex*{I_GB \cap \tors{B^G}}{I_GB\cap \tors{B^D}}}.
\]
To prove the above equality, observe that $\Sigma$ acts as $-1$ on the quotient 
\[
\frac{B[N_G]\cap\bigl(B^{\Sigma}+B^{\Sigma'}+B^{G}\bigr)}{B[N_G]\cap\bigl(B^{\Sigma}+B^{\Sigma'}\bigr)}
\]
by \eqref{eq:quozcoom}. In particular we have
\begin{align*}
\frac{B[N_G]\cap\bigl(B^{\Sigma}+B^{\Sigma'}+B^{G}\bigr)}{B[N_G]\cap\bigl(B^{\Sigma}+B^{\Sigma'}\bigr)}& = \left(\frac{B[N_G]\cap\bigl(B^{\Sigma}+B^{\Sigma'}+B^{G}\bigr)}{B[N_G]\cap\bigl(B^{\Sigma}+B^{\Sigma'}\bigr)}\right)^-\\
& = \frac{\left(B[N_G]\cap\bigl(B^{\Sigma}+B^{\Sigma'}+B^{G}\bigr)\right)^-}{\left(B[N_G]\cap\bigl(B^{\Sigma}+B^{\Sigma'}\bigr)\right)^-}\\
\text{(by \eqref{eq:BSigma-})}& = \frac{\left(B[N_G]\cap\bigl(B^{\Sigma}+B^{\Sigma'}+B^{G}\bigr)\right)^-}{\left(I_GB\right)^-}.
\end{align*}
Now note that $B^{\Sigma}+B^{\Sigma'} = I_GB + B^\Sigma$ by \eqref{eq:BSigma'}, so that 
\begin{align*}
B[N_G]\cap\bigl(B^{\Sigma}+B^{\Sigma'}+B^{G}\bigr) &= B[N_G]\cap\bigl(I_GB +B^{\Sigma}+B^{G}\bigr)\\
&=I_GB +  \bigl(B[N_G]\cap\bigl(B^{\Sigma}+B^{G}\bigr)\bigr)
\end{align*}
and taking minus parts we obtain
\begin{align*}
\bigl(I_GB +  \bigl(B[N_G]\cap(B^{\Sigma}+B^{G})\bigr)\bigr)^-&=\left(I_GB\right)^- +  \bigl(B[N_G]\cap(B^{\Sigma}+B^{G})\bigr)^-\\
&=\left(I_GB\right)^- +  B[N_G]\cap\left(B^{G}\right)^-\\
&=\left(I_GB\right)^- +  \left(\tors{B^{G}}\right)^-\\
&=\left(I_GB+\tors{B^{G}}\right)^-.
\end{align*}
Here we have used that for uniquely $2$-divisible $\Sigma$-modules $B_1$ and $B_2$, we have $(B_1+B_2)^- = \frac{1-\sigma}{2}(B_1+B_2) = \frac{1-\sigma}{2}B_1+ \frac{1-\sigma}{2}B_2 = B_1^- + B_2^-$. Therefore
\begin{align*}
\frac{B[N_G]\cap\bigl(B^{\Sigma}+B^{\Sigma'}+B^{G}\bigr)}{B[N_G]\cap\bigl(B^{\Sigma}+B^{\Sigma'}\bigr)} &= \frac{\left(I_GB+\tors{B^{G}}\right)^-}{\left(I_GB\right)^-}\\
&\cong\frac{\left(\tors{B^{G}}\right)^-}{\left(I_GB\cap \tors{B^G}\right)^-}\\
&=\frac{\tors{B^G}/\tors{B^D}}{(I_GB\cap \tors{B^G})/(I_GB\cap \tors{B^D})}
\end{align*}
and this concludes the proof of the lemma.
\end{proof}

Let $B$ be a finitely generated $\Z$-module and let $m$ be a positive natural number. Then we have
\[
\gorder*{\tors[m]{B}}\cdot m^{\rank_\Z(B)}= \gindex{B}{mB}.
\]
Indeed, let $\pi\colon B\to \overline{B}=B/\tor_{\Z}(B)$ be the projection map. It induces a commutative diagram with exact rows
\[
\xymatrix{
0\ar@{->}[0,1]&\tors[m]{B}\ar@{->}[0,1]&B\ar@{->}[0,1]^{m}\ar@{->>}[1,0]^{\pi}&B\ar@{->}[0,1]\ar@{->>}[1,0]^{\pi}&B/mB\ar@{->}[0,1]\ar@{->}[1,0]^{\pi_m}&0\\
&0\ar@{->}[0,1]&\overline{B}\ar@{->}[0,1]^{m}&\overline{B}\ar@{->}[0,1]&\overline{B}/m\overline{B}\ar@{->}[0,1]&0
}
\]
where $\pi_m$ is the map induced by $\pi$ on $B/mB$. The snake lemma gives an exact sequence of finite abelian groups
\[
0\longrightarrow \tors[m]{B}\longrightarrow\tor_\Z(B)\stackrel{m}{\longrightarrow}\tor_\Z(B)\longrightarrow\kernel{\pi_m}\longrightarrow0.
\]
On the other hand we have the tautological exact sequence 
\[
0\longrightarrow \kernel{\pi_m}\longrightarrow B/mB \longrightarrow\overline{B}/m\overline{B}\longrightarrow 0.
\]
Taking cardinalities we obtain
\begin{equation}\label{eq:rankvstorsion}
\gorder*{\tors[m]{B}}=\gorder*{\kernel{\pi_m}} =\frac{\gindex{B}{mB}}{\gindex{\overline{B}}{m\overline{B}}} =\gindex{B}{mB}m^{-\mathrm{rk}_\Z(B)},
\end{equation}
as claimed.

\begin{proof}[Proof of Proposition \ref{prop:unitindexcohom}]
Combining Lemmas \ref{lemma:unitindexppower}, \ref{lemma:firstfactor} and \ref{lemma:secondfactor} we get
\begin{align*}
\gindex*{B}{B^{\Sigma}+B^{\Sigma'}+B^{G}} &= \gindex*{N_G B}{N_G(B^{\Sigma}+B^{\Sigma'}+B^{G})}\cdot\gindex*{B[N_G]}{B[N_G]\cap (B^{\Sigma} + B^{\Sigma'} + B^{G})}\\
&=\frac{\hh*{0}{D}{\pd{B}}\cdot \gindex*{B^G}{q(B^G)}}{\hh*{0}{G}{\pd{B}}\cdot \gindex*{B^D}{q(B^D)\vphantom{B^G}}}\frac{\hhN*{1}{G}{\pd{B}}\cdot \gindex{I_GB\cap \tors{B^G}}{I_GB\cap \tors{B^D}}}{\hh*{-1}{D}{\pd{B}}\cdot \gindex{\tors{B^G}}{\tors{B^D}}}\\
&= \frac{\gindex*{B^G}{q(B^G)}}{\gindex*{B^D}{q(B^D)\vphantom{B^G}}\cdot\gindex{\tors{B^G}}{\tors{B^D}}}\frac{\hh*{0}{D}{\pd{B}}}{\hh*{-1}{D}{\pd{B}}}\frac{\gindex*{I_GB\cap \tors{B^G}}{I_GB\cap \tors{B^D}}}{\herbr{G}{B}}.
\end{align*}
Now observe that, in light of \eqref{eq:rankvstorsion}, we have
\[
\frac{\gindex*{B^G}{q(B^G)}}{\gindex*{B^D}{q(B^D)\vphantom{B^G}}}=q^{\rank_\Z(B^G)-\rank_\Z(B^D)}\gindex*{\tors{B^G}}{\tors{B^D}}.\qedhere
\]
\end{proof}

\bibliography{biblio}
\bibliographystyle{amsalpha}

\end{document}